\documentclass[letterpaper]{amsart}
\usepackage{amsmath}
\usepackage{amsthm}
\usepackage{amssymb}
\usepackage{amsaddr}
\usepackage{graphicx}
\usepackage{tikz}
\usepackage{color}

\newtheorem{lemma}{Lemma}
\newtheorem{theorem}[lemma]{Theorem}
\newtheorem{definition}[lemma]{Definition}
\newtheorem{corollary}[lemma]{Corollary}
\theoremstyle{remark}
\newtheorem{remark}[lemma]{Remark}
\title{Classification of real rational knots of low degree in the 3-sphere}
\author{Shane D'Mello}
\address{Department of Mathematics\\Stony Brook University\\Stony Brook, NY 11794, USA}
\email{shane@math.sunysb.edu}
\begin{document}
\maketitle
\begin{abstract}
  \noindent In this paper we classify, up to rigid isotopy, non-singular real rational curves of degrees less than or equal to $6$ in a quadric homeomorphic to the 3-sphere. We also study their connections with rigid isotopy classes of real rational knots in $\mathbb{RP}^3$.
\end{abstract}
  
\section{Introduction}
\subsection{Background}
In the topology of real algebraic varieties the most traditional objects of study were either non-singular varieties (like surfaces) or subvarieties of co-dimension one (curves in the 2-space or surfaces in the 3-space). There were few papers devoted to the topology of curves in the 3-space~\cite{bjorklund}~\cite{encomplex}, and there the ambient 3-space was the real projective space $\mathbb{RP}^3$. From a purely topological point of view, the curves are knots in $\mathbb{RP}^3$. There were also papers devoted to polynomial knots in $\mathbb{R}^3$~\cite{polynomialexamples}~\cite{polynomialcompleting}.

In topology,  the more traditional objects are classical knots which are curves in the 3-sphere $S^3$. The sphere $S^3$ appears as the set of real points of algebraic 3-folds. The simplest of them is the quadric in $\mathbb{RP}^4$ defined by the equation $x_1^2+x_2^2+x_3^2+x_4^2=x_0^2$. Following notations used in topology, we denote this quadric by $S^3$.

In this paper we will study the simplest real algebraic knots in $S^3$: non-singular real rational curves of low degrees. A curve $C\subset S^3$ is rational if there exists a regular map $k:\mathbb{RP}^1 \to \mathbb{RP}^4$ which defines an isomorphism $k:\mathbb{RP}^1 \to C$. Such a map $k$ is defined by formulas ($[s:t]\to [k_0(s,t):k_1(s,t):k_2(s,t):k_3(s,t):k_4(s,t)]$) where $k_0,\ldots,k_4$ are real homogeneous polynomials of degree~$d$ with no common root. 

A map $k$ is not unique, but $C$ defines $k$ up to real projective transformations of the source $\mathbb{RP}^1$. We will call $k$ a (rational) \textbf{parametrization} of $C$.

The same polynomials define a regular map $\mathbb{C}k:\mathbb{CP}^1\to\mathbb{CP}^4$, a parametrization of the set $\mathbb{C}C$ of complex points of $C$. 

A rational curve $C$ is called a \textbf{real rational knot} if $C$ has no singular points (neither real nor complex: i.e. its parametrization $\mathbb{C}k:\mathbb{CP}^1\to\mathbb{CP}^4$ has no double points and at least one of the partial derivatives is non-zero).

A rational curve $C\in S^3$ with singularities is called a singular real rational knot. Real rational knots may be considered from the perspectives of knot theory where the major equivalence relation comes from isotopy. Recall that knots $k_1,\ k_2\in S^3$ are called isotopic if they can be included in a one parameter continuous family $k_t\in S^3$ of classical knots. 

In our real algebraic setup, where all knots are real algebraic curves, the notion of isotopy is modified accordingly to a one-parameter family of non-singular real algebraic curves, that is a path in the space of such curves. 
Such a path is called a \textbf{rigid isotopy}. Observe that a rigid isotopy of real algebraic knots can be lifted to a continuous family of parametrizations. In particular, the degree of the knots (which is the degree polynomials forming the parametrization) is preserved under a rigid isotopy. The notion of rigid isotopy can be defined in a similar way in a more general situation and real algebraic varieties of several classical types were classified up to rigid isotopy (see for example,~\cite{survey}).

In ~\cite{encomplex}, Viro defined the \textbf{encomplexed writhe number} for real rational knots in $\mathbb{RP}^3$. It is a rigid isotopy invariant. The definition can be used to define the writhe number of a knot in the sphere as follows: given a knot $C$ in the $S^3\subset\mathbb{RP}^4$, choose a point  of projection $p$ on the sphere which is not on the knot and project the knot to $\pi_p(C)$ in $\mathbb{RP}^3$. The encomplexed writhe number of $C$ is defined as the writhe number of $\pi_p(C)$. Note that this is independent of the point of projection because the image of knots under two different such projections are rigidly isotopic in $\mathbb{RP}^3$ and so they will have the same writhe number.

\subsection{Main results}
An odd degree rational curve in $\mathbb{RP}^4$ must intersect a generic linear hypersurface in an odd number of real points to preserve the mod $2$ intersection number. This is impossible for odd degree knots contained in the sphere because there are real planes disjoint from the real sphere. Therefore:
\textit{all real rational knots in the sphere are of even degree.} We will prove the following rigid isotopy classification for knots of degrees $2$, $4$, and $6$ in $S^3$:
\begin{theorem}
  The rigid isotopy classes of knots of degrees less than~$6$ can be described as follows:
  \begin{enumerate}
  \item Degree $2$: there is one rigid isotopy class; it has writhe number $0$. All knots are topologically isotopic to the unknot.
  \item Degree $4$: there are two rigid isotopy classes; they have writhe numbers $\pm 2$. Mirror images of each knot lie in separate classes. Each of them are topologically isotopic to the unknot. 
  \item Degree $6$: there is one rigid isotopy class with each of the writhe numbers $0$, $\pm 2$, or $\pm 4$. Their smooth isotopy classes may be described as follows:
    \begin{enumerate}
	  \item writhe $0$: unknot
	  \item writhe $2$: unknot
	\item writhe $-2$: unknot
	\item writhe $4$: trefoil
	\item writhe $-4$: trefoil
	\end{enumerate}
      \end{enumerate}
    \end{theorem}

In fact for degrees  $2$ and $4$ we will prove a stronger result that knots in the same rigid isotopy class are projectively equivalent: any two of them can be transformed from one to the other by a projective linear transformation in $\mathbb{RP}^4$ that preserves $S^3$. 

%% Change this part
Besides merely classifying them, we will also relate them to knots in the projective space. Indeed, we will use the classification of knots in the projective space to derive our classification, although we will need to extend the classification in the projective space to pairs of knots and planes that intersect in a special way. 

To construct examples, we will note that for each degree $d$, there are $(m,d/2)$-torus knots for any $m$ which is coprime to $d/2$.  We will also observe that Bj{\"o}rklund's method~\cite{bjorklund} of constructing real rational knots can be lifted to the sphere to provide an exact counterpart to his method of construction. This will at first seem to contradict the classification of curves in $S^3$ which does not allow the figure eight knot to be realized by a degree~$6$ curve. However, section~\ref{impossible} will demonstrate why we need to be careful when constructing curves in $S^3$ because we are not permitted to use general conics (as Bj{\"o}rklund had used in $\mathbb{RP}^3$ setting), only circles. Indeed, the trefoil in the sphere cannot be constructed by perturbing the configuration that Bj{\"o}rklund used, but uses a different configuration.

Finally, in section~\ref{sec:deg8}, we will prove that a degree~6 real rational knot in $\mathbb{RP}^3$ which intersects a plane in a maximum of two real points, is rigidly isotopic to a stereographic projection of a real rational knot in the 3-sphere with one double point.  This will establish a connection on the level of isotopy between real rational knots of degree~6 in $\mathbb{RP}^3$ and real rational knots of degree~8 in $S^3$ with one double point.

%%%% Change ``later''
Bj{\"o}rklund~\cite{bjorklund} had classified real rational knots in $\mathbb{RP}^3$ up to degree~$5$. It is possible to partially add to that classification by classifying singular knots of degree $6$ in $\mathbb{RP}^3$  with four double points: \textit{up to mirror reflections, there is a bijection between the rigid isotopy classes of degree 6 rational knots with four real double points and chord diagrams with at most four chords}. This method of reducing the rigid isotopy classification to a combinatorial one of classifying chord diagrams, is a straightforward generalization of the rigid isotopy classification of degree~4 curves in the plane shown in~\cite{mine}, by the same author. However, in the case of planar quartics, the number of double points need not be forced, since generic planar quartics have three double points. Since degree~8 curves in $\mathbb{RP}^4$ with 5 double points, lie on the 3-sphere, this method also generalizes to a partial classifcation of those degree~8 knots in the 3-sphere with 5 double points.

%%% method?
\subsection{Outline of the method}
There are several difficulties when trying to classify real rational knots in $S^3$ up to rigid isotopy. In~\cite{bjorklund}, linear transformations were used to simplify knots in $\mathbb{RP}^3$; for knots in $S^3$, we can only use linear transformations that preserve the sphere. This does not help except in the simplest cases like degrees~$4$. In fact the method of classification of degree~$4$ knots in $S^3$ is very similar to the approach used in \cite{bjorklund} to classify rational curves of degree~$3$ in $\mathbb{RP}^3$, which suggests a connection between knots in $S^3$ and knots in $\mathbb{RP}^3$ of lower degree. 

The connection is owing to the stereographic projection. If one considers knots with one singularity, the degree~$d$ knots in $S^3$ with the north pole as its only double point, relate to degree~$d-2$ knots in $\mathbb{RP}^3$ which intersect the plane at infinity in $d-2$ points, $d-4$ of which lie on a conic of signature $0$. More generally we may consider pairs of knots and planes as described in the next two  paragraphs.

Following the notation in~\cite{bjorklund}, let $\mathcal{K}_d$ denote the space of real rational knots of degree~$d$ in $\mathbb{RP}^3$.  $\mathbb{P}\mathrm{Gr}(3,\mathbb{R}^4)$ is the grassmanian of 2-planes in $\mathbb{RP}^3$. Consider the space of pairs $(C,X)\in\mathcal{K}_d\times\mathbb{P}\mathrm{Gr}(3,\mathbb{R}^4)$, where $C$ is a real rational knot of degree~$d$ in $\mathbb{RP}^3$ and $X$ is a real plane in $\mathbb{RP}^3$. Denote by $\mathcal{P}_{d,m}$ the subspace of those pairs such that $\mathbb{R}C$ and $\mathbb{R}X$ intersect in exactly $m$ real points, where $m=0$ or $m=2$. 

Define $\mathcal{P'}_{d,m}\subset\mathcal{P}_{d,m}$ to be the pairs $(C,X)$, such that $d-4$ points of intersection of $\mathbb{C}C$ and $\mathbb{C}X$ lie on a conic in $X$ of signature  $0$. 

In theorem~\ref{maincorrespondence} of section~\ref{stereographic}, we will prove that the space of rational knots of degree~$d$ in $S^3$ with exactly one double point, is a double covering of $\mathcal{P'}_{d-2,m}$ where $m=0$ if the double point is solitary, and $m=2$ if it is not. The two curves in each fibre of the double covering are rigidly isotopic to the mirror images of each other.  

It will be difficult to extend an isotopy of knots in $\mathbb{RP}^3$ to an isotopy of pairs in $\mathcal{P'}_{d,m}$ because of the condition imposed on the intersection of the knot and the plane. However, in section~\ref{retract}, we will show that when $d\leq 6$, then $\mathcal{P'}_{d,m}$ is a strong deformation retract of $\mathcal{P}_{d,m}$.  This simplifies the restriction on the intersection, however maintinaining the right number of real intersection points is still difficult. For degree~$6$, we will be able to overcome this difficulty because they relate to degree~$4$ knots in $\mathbb{RP}^3$, which always lie on a quadric.  Knowing the bi-degree of the degree~4 knot on the quadric will make it easier to classify the pairs of knots and planes which intersect in a fixed number of real points. 

While this will classify degree~$6$ knots with one double point, we obtain the non-singular knots by perturbing these. By using the Borel-Moore homology, we will that the number of rigid isotopy components cannot exceed 5. In order to obtain this bound, we will be faced with another difficulty: that the space of rational knots of degree~$6$ in $S^3$ is a connected manifold. This problem is trivial for knots in a projective space, since the space of knots of a fixed degree in $\mathbb{RP}^3$ is itself a projective space. We will solve it for rational knots in $S^3$ in section~\ref{manifold}. Finally, we will construct 5 rational knots of degree~6 in the sphere with different writhe numbers, to obtain a representative for each of the 5 rigid isotopy classes. 

\section{Degrees 2 and 4}
\subsection{Degree 2 knots}
Degree~$2$ knots in the sphere also lie on a two dimensional plane in $\mathbb{RP}^4$ because any three points on the curve define a pencil of planes; it has to lie on the intersection of the planes from the pencil because it intersects each plane in more than $2$ points. The intersection is a two dimensional plane on which our knot lies and so it is a planar conic which is non-empty. It is well known that non-empty non-singular planar conics are projectively equivalent and so the knot is rigidly isotopic to the unknot.   

\subsection{Degree 4 knots} 
\begin{lemma}
  \label{deg4lin}
  A degree~$4$ knot in $S^3$ lies on a linear hypersurface of $\mathbb{RP}^4$ if and only if it has a double point. 
\end{lemma}
\begin{proof}
Degree~$4$ knots in the sphere  with double points, lie on a linear hypersurface. This is because any double point on the knot along with three other points define a  linear hypersurface in $\mathbb{RP}^4$ which intersects the knot in five points (counting multiplicity, because of the double point). So the singular knot is forced to lie on the linear hypersurface. 

Conversely, if a curve lies on a linear hypersurface in $\mathbb{RP}^4$, then it must be singular. This can be seen by projecting it  to $\mathbb{RP}^3$ via the stereographic projection, from a point on the linear hypersurface but not on the curve. It will project to a degree~$4$ curve on a plane. The complexification of this rational curve must have 3 double points because any planar rational curve of degree~$d$ must have $\frac{(d-1)(d-2)}{2}$ double points. Since the curve is real, any imaginary double points must occur in pairs. Therefore, at least one double point is real and the original curve is not a knot. 
\end{proof}

Observe that changing the parametrization of a knot $C$ parametrized by $k: \mathbb{RP}^1\to S^3 \subset \mathbb{RP}^4$ means composing $k$ by a transformation of $\mathbb{RP}^1$, namely a transformation in $\mathbb{P}GL(2,\mathbb{R})$ which has two connected components that correspond to transformations that either retain or reverse the orientations respectively. Therefore changing the parametrization provides a knot which is either the same or reverses its orientation. 

Projective transformations that preserve the $3$-sphere form the projectivization of the transformations that preserve the bilinear form $x_1^2+x_2^2+x_3^2+x_4^2-x_0^2$, namely the indefinite orthogonal group $O(4,1)$. $O(4,1)$ has four components but after projectivizing it has two components. One component contains the identity whereas the other component contains the reflection. Denote the projectivization of $O(4,1)$ by $\mathbb{P}O(4,1)$. Given any curve $C$ in $S^3$, its image under a transformation $A\in \mathbb{P}O(4,1)$ is either isotopic to itself or to its mirror image. 

Recall that a finite set of points in $\mathbb{RP}^n$ is said to be in \textit{general position} if there is no subset with $n+1$ points that all lie on a linear hypersurface. The following lemma will show that any projective transformation will do:
\begin{lemma}
 A transformation in $\mathbb{P}GL_5(\mathbb{R})$ that takes at least $14$ points in general position in $S^3$ to points within $S^3$ is in $\mathbb{P}O(4,1)$.
\end{lemma}
\begin{proof}
 The images of the points are also in general position because the inverse of the transformation would pull back any linear relations satisfied by the images. The images lie in the image of the sphere (another quadric), but are given to be in the original sphere. The sphere and its image must coincide because $14$ points in  general position uniquely define a quadric in $\mathbb{RP}^4$. So the transformation preserves $S^3$.
\end{proof}

\begin{theorem}
Any two degree $4$ real rational knots in the $3$-sphere are projectively equivalent. Therefore, there are two rigid isotopy classes which are mirror images of each other. 
\end{theorem}

\begin{proof}
  Let $k_1, k_2: \mathbb{RP}^1\to S^3 \subset \mathbb{RP}^4$ each denote the parametrization of two knots $C_1$ and $C_2$ of degree $4$.  Distinct points on $C_1$ or $C_2$ are now in general position in $\mathbb{RP}^4$ because a linear hypersurface of $\mathbb{RP}^4$ containing more than five of these points would contain the curve for intersecting it in more than four points. This would contradict lemma~\ref{deg4lin}.

  If the $i$th coordinate of a parametrization $k$ of a real rational knot is the polynomial $\sum_{j=0}^d \alpha_j^i s^jt^{d-j}$, then let $A_k$ denote the matrix of coefficients $(\alpha_i^j)$ of $k$. Note that knot defined by $k$ not lying on a linear hypersurface of $\mathbb{RP}^4$ is equivalent to $A_k \in \mathbb{P}GL_5(\mathbb{R})$. Therefore, $A_{k_1} A_{k_2}^{-1} \in \mathbb{P}GL_5(\mathbb{R})$. Observe that $A_{k_1} A_{k_2}^{-1}k_2=k_1$. Therefore, they are projectively equivalent.
  
If $A_{k_1} A_{k_2}^{-1}$ lies in the component of $\mathbb{P}O(4,1)$ that contains the identity, then the knots defined by $k_1$ and $k_2$ are rigidly isotopic, otherwise $k_1$ is isotopic to the mirror image of $k_2$. 
\end{proof}

%%We will now prove this by relating curves in $S^3$ to curves in $\mathbb{RP}^3$ via the stereographic projection. This will strengthen the above result by showing that the orientation degree $4$ knots do not change their isotopy type on reversal of orientation. 

\subsection{The connection between real rational knots in $S^3$ and real rational knots in $\mathbb{RP}^3$}
\label{stereographic}
The idea of using the fact that non-planar curves are non-singular, is borrowed from Bj{\"o}rklund's paper~\cite{bjorklund}. The behaviour of degree~$4$ knots in $S^3$ is very similar to degree~$3$ knots in $\mathbb{RP}^3$. This connection will be made explicit by using the stereographic projection. Since we are concerned with the stronger notion of rigid isotopy, we will need to keep a track of the intersection of the curve with the blow up of the point of projection.

Consider the sphere $S^3\subset\mathbb{RP}^4$ and a point $p$ on $S^3$. A linear hypersurface in $\mathbb{RP}^4$ which is disjoint from $S^3$ can be treated as a copy of $\mathbb{RP}^3$. The projection map $\pi_p:S^3\setminus p\to\mathbb{RP}^3$ is defined by mapping any point $\alpha$ to the intersection of the plane with the unique line joining $p$ and $\alpha$. The line is unique because any line can transversally intersect a quadric in only two points, one of which is $p$.

The projection can be extended to the complexification $\mathbb{C}S^3\subset\mathbb{CP}^4$. If a point $\alpha$ on $S^3$ is such that the line joining it with $p$ is not tangential to $S^3$, then the line does not intersect any other points of $S^3$. Therefore, $\pi_p$ is a bijection on the set $S^3\setminus\mathbb{C}T_p$, where $T_p$ is the tangent plane to $S^3$ at $p$. It is a rational map on $S^3$ that blows up $p$ to a plane $X_p$ in $\mathbb{RP}^3$.

$\mathbb{C}_p \cap \mathbb{C}S^3$ is a quadric with a singularity at $p$. It is therefore a complex 2-dimensional cone with an apex at $p$. This cone intersects the plane of projection $\mathbb{RP}^2$ in an empty conic lying on $X_p$. Recall that an empty conic  is a conic defined by a polynomial with real coefficients but which does not have any real zeros. 

The complexification $\mathbb{C}C$ of any degree~$d$ rational curve $C$, parametrized by $k$, intersects $\mathbb{C}T_p$ in $d$ points (counting multiplicity). These $d$ points of intersection must lie on the cone defined by $\mathbb{C}_p \cap \mathbb{C}S^3$. If $p$ does not lie on the curve $C$, the closure of $\pi_p(\mathbb{C}C)\setminus \mathbb{C}T_p$ intersects $\mathbb{C}X_p$ in $d$ points that all lie on the empty conic which is the projection of the cone $\mathbb{C}_p \cap \mathbb{C}S^3$. We define the image $\pi_p(C)$ of the curve $C$ under the rational map $\pi_p$ as the rational curve defined by the closure of $\pi_p(\mathbb{C}C\setminus \mathbb{C}T_p)$.  Since the image intersects the plane $X_p$ in $d$ points, the degree of this rational curve is $d$. $\pi_p$ is a bijection between real rational curves in $S^3$ of degree~$d$ and real rational curves in $\mathbb{RP}^3$ of degree~$d$ which intersect the blow-up of $p$ in points that lie on an empty conic. 

By choosing the point of projection to be on the curve $C$, the degree of $\pi_p(C)$ can be made to drop. If the curve $C$ passes through $p$, then $T_p$ intersects $C$ at $p$ with multiplicity $2$, then $d-2$ of them intersect $T_p$ in imaginary points. The closure of $\pi_p(\mathbb{C}C\setminus \mathbb{C}T_p)$ intersects $\mathbb{C}X_p$ in $d-1$ points and is therefore a curve of degree~$d-1$. $d-2$ of these points of intersection lie on the empty conic which is the projection of the cone $\mathbb{C}_p \cap \mathbb{C}S^3$. $\pi_p$ is a bijection between real rational curves in $S^3$ of degree~$d$ and real rational curves in $\mathbb{RP}^3$ of degree~$d-1$ which intersects the blow-up of $p$ in $d$ points, $d-1$ of which lie on an empty conic. 

The branch of the curve around $p$ gets mapped to a branch that intersects the plane at infinity once. This point of intersection is determined by the direction of the tangent line to the curve at $p$.

It will prove more useful to  consider the case where the point $p$ is a double point. 
\begin{theorem}
  \label{maincorrespondence}
  The space of degree~$d$ knots in $S^3$ with one double point is a double covering of $\mathcal{P}'_{d-2,k}$, where $k=0$ if the double point is solitary, and $k=2$ if its not. Given a pair $(C,X) \in \mathcal{P}_{d-2,k}$, the two lifts of it are related: they are each isotopic to the mirror image of the other.
\end{theorem}
\begin{proof}
  If $p$ is a double point of the curve $C$, then $T_p$ intersects two branches of $C$ at $p$, each with multiplicity $2$. In that case, $d-4$ of them intersect $T_p$ in imaginary points. The closure of $\pi_p(\mathbb{C}C)\setminus \mathbb{C}T_p$ intersects $\mathbb{C}X_p$ in $d-2$ points, and is therefore of degree~$d-2$. $d-4$ of them lie on the empty conic which is the projection of the cone $\mathbb{C}_p \cap \mathbb{C}S^3$. This will help us to relate rational knots of degree $d$ with one double point, to rational curves of degree $d-2$ in $\mathbb{RP}^3$ that intersect the blow up of the double point in $d-2$ points, $d-4$ of which lie on an empty conic. 
\end{proof}

\subsubsection{Simplifying the hypothesis for degree~$6$}
The main difficulty is in maintaining the right intersection for between a knot and a plane belonging to a pair in $\mathbb{P}_{4,k}$. Recall that we had defined the space $\mathcal{P}'_{4,k}$ to be the space of those pairs $(C,k)$, where $C$ intersects $k$ in $k$ real points, but the imaginary points of intersection need not lie on the empty conic. 

\begin{theorem}
  \label{retract}
  $\mathcal{P}'_{4,k}$ is a deformation retract of $\mathcal{P}_{4,k}$, where $k=0$ or $k=2$. In other words: there is a retraction map $r:\mathcal{P}_{4,k}\to \mathcal{P}'_{4,k}$ which is homotopic to the identity.
\end{theorem}
\begin{proof}
  For each line we will need a fixed real point that does not lie on it. But it is possible to define a continuous map $\theta$ from the space of all real lines at the plane at infinity to the space of all real points at the plane at infinity such that each line is taken to a point that does not lie on it. 

  Define a homotopy of maps $[0,1]\times\mathcal{P'}_{4,k}\to\mathcal{P}_{4,k}$ as follows: given a curve $C$ in $\mathcal{P}_{4,k}$, it intersects the plane at infinity in a conjugate pair $z$ and $\bar{z}$ which will define a real line $l$. $\mathbb{C}l$ intersects the standard empty conic in another conjugate pair $w$ and $\bar{w}$ where  we let $w$ denote the point that is in the same component of $\mathbb{C}l\setminus\mathbb{R}l$ as $z$. By the previous paragraph, $\theta(l)$ is a point in the plane at infinity. Define $T_t$ to be the unique real linear transformation that fixes the real points $[1:0:0:0]$ and $\theta(l)$, and maps $z\to (1-t)z+tw$ and $\bar{z}\to (1-t)\bar{z}+t\bar{w}$. Observe that $T_0$ is the identity map, and  $T_1$ takes $z\to w$ and $\bar{z}\to\bar{w}$. Therefore the map $r:\mathcal{P}_{4,k}\to\mathcal{P}'_{4,k}$ defined by $r(C):=T_1(C)$ is a retract which is homotopic to the identity. 
\end{proof}

Therefore, to prove that two degree~$6$ curves in $S^3$, each with one double point, are isotopic, we only need to find a path between their stereographic projections in the space $\mathcal{P}_{4,k}$.
\begin{remark}
  We can use the stereographic projection to classify the degree~4 knots in $S^3$ by connecting them with degree~3 knots in $\mathbb{RP}^3$. 

  A degree~$4$ knot in $S^3$ can be projected by a point on the knot to a knot in $\mathbb{RP}^3$ which intersects the plane at infinity in one real and a pair of conjugate imaginary points. The imaginary points lie on the empty quadric. Any degree~$3$ curve in $\mathbb{RP}^3$ intersects the plane at infinity in one real and a pair of conjugate imaginary points. The conjugate imaginary points can always made to lie on the standard empty conic and therefore any degree~$3$ curve in $\mathbb{RP}^3$ may be lifted to $S^3$ by the stereographic projection. Therefore, the isotopy classes of degree~$4$ knots in $S^3$ correspond to the isotopy classes of degree~$3$ knots in $\mathbb{RP}^3$. Since there are two isotopy classes of degree~$3$ knots in $\mathbb{RP}^3$, there are two isotopy classes of degree~$4$ knots in $S^3$~\cite{bjorklund}. \end{remark}

\section{Degree $6$}
To classify degree~$6$ knots in $S^3$ with one double point, we can relate them to pairs of degree~$4$ knots and planes in $\mathbb{RP}^3$. Bj{\"o}rklund~\cite{bjorklund} had classified degree 4 curves by first classifying the singular ones and then perturbing them. But this cannot be used to extend this to the space of pairs of knots and planes, so we will consider the fact that degree~$4$ curves lie on a quadric. But first we will show how theorem~\ref{maincorrespondence} can be used to obtain the degree~$4$ curves with one double point.

Change bases so that the double point is $[1:0:0:0]$ as the images of either $[0:1]$ and $[1:0]$, or $[1:i]$ and $[1:-i]$, depending on whether it is real or imaginary. Project from this to the plane defined by $x_0=0$. This projection will produce only finitely many double points because the knot is on a quadric and so a double points can only occur when its inverse under the projection is a line on the quadric that passes through $[1:0:0:0]$. At most two lines can pass through $[1:0:0:0]$, namely the ones from the ruling of the quadric. 

As before the projection is a curve degree $2=4-2$, because the common factors of either $st$ or $s^2+t^2$ will pull out. Change the basis so that this degree 2 curve is the standard $[s:t]\to [s^2:st:t^2]$. Since the knot lies on the cone over this curve with the apex as the double point, it must be of the form $[s^3t:s^2t^2:st^3:p_4[s:t]]$ or $[s^2(s^2+t^2):st(s^2+t^2):t^2(s^2+t^2):p_4[s:t]]$ where $p_4[s:t]$ is a polynomial of degree 4. 

In the case where they are real double points, this polynomial may be reduced to either $s^4+t^4$, $s^4-t^4$, $-s^4+t^4$, or $-s^4-t^4$ because all the other coefficients may be brought to zero without forming a double point as the first three coordinates are distinct anyway. We only need to avoid $s$ or $t$ dividing it or else the degree will drop; the presence of both $s^4$ and $t^4$ prevent this. 

In the case where they are solitary double points, we only need to avoid multiples of $s^2+t^2$ which form a co-dimension two hypersurface in the space of polynomials of degree 4. This can therefore be reduced to a standard form of $[s^2(s^2+t^2):st(s^2+t^2):t^2(s^2+t^2):s^4]$.
Therefore the walls are:
\[[s^3t:s^2t^2:st^3:s^4+t^4]\]
\[[s^3t:s^2t^2:st^3:s^4-t^4]\]
\[[s^3t:s^2t^2:st^3:-s^4+t^4]\]
\[[s^3t:s^2t^2:st^3:-s^4-t^4]\]
\[[s^2(s^2+t^2):st(s^2+t^2):t^2(s^2+t^2):s^4]\]

These were precisely the walls considered by Bj{\"o}rklund~\cite{bjorklund}, and he also observed that the first and the third are rigidly isotopic. By perturbing them he showed that \textit{knots of degree $4$ were either the unknot and its image (with writhes $+1$ or $-1$) or the two crossing knot  and its mirror image (with writhes $+3$ or $-3$)}.

\subsection{Classification of $\mathcal{P}_{4,k}$}

Degree~4 curves in $\mathbb{RP}^3$ must lie on a quadric surface because we can always find a quadric surface passing through any 9 points on the curve. The curve can intersect the quadric in more than 8 points only if it lies on it. To classify $\mathcal{P}_{4,k}$, it will prove essential to understand how the knots of degree~$4$ lie on this quadric surface. 

%% In the real case/ in the complex case
In the complex case, there is only one non-singular quadric. In the real case, there are two non-singular quadrics, which are the sphere and hyperboloid. 

By a change of coordinates, the hyperboloid can be defined by as the zero set of the polynomial $z_0z_3-z_1z_2=0$. It can also be realized as the image of the embedding, $\theta:\mathbb{RP}^1\times \mathbb{RP}^1\to \mathbb{R}Q\subset \mathbb{RP}^3$, which is defined explicitly by the map \[([x_0:x_1],[x_2:x_3])\to [x_0x_2:x_0x_3:x_1x_2:x_1x_3]\] This map can be extended to the complexification, which is known as the Segre embedding. 

%%%Fix this
For a fixed $[\alpha:\beta]$, the image of $([\alpha:\beta],[x_2:x_3])$ is a line $L_{[\alpha:\beta]}$. Similarly, for a fixed $[\alpha:\beta]$, the image of $([x_0:x_1],[\alpha:\beta])$ is a line $L'_{[\alpha:\beta]}$. The lines $L_{[\alpha:\beta]}$ form a family of rulings $F$, while the lines $L'_{[\alpha:\beta]}$ form a family of rulings $F'$. If two lines are from the same family they are disjoint, and if they are different families they intersect in one point. The homology classes $[\mathbb{R}L]$ and $[\mathbb{R}L']$ with representatives $L$ and $L'$ from different families, generate $\mathrm{H}_1(\mathbb{R}Q)=\mathbb{Z}\oplus\mathbb{Z}$. Similarly the homology classes $[\mathbb{C}L]$ and $\mathbb{C}[L']$ with representatives $L$ and $L'$ from different families, generate $\mathrm{H}_2(\mathbb{C}Q)=\mathbb{Z}\oplus\mathbb{Z}$

By a change of coordinates, a real sphere in the projective space can be defined to satisfy the equation $x_1^2+x_2^2=x_0^2-x_3^2$. The rulings on it are imaginary and can be realized as follows. Each side of $x_1^2+x_2^2=x_0^2-x_3^2$ can be factored to give $(x_1+ix_2)(x_1-ix_2)=(x_0-x_3)(x_0+x_3)$. So the bijection $[x_0:x_1:x_2:x_3]\to [x_1+ix_2:x_0-x_3:x_0+x_3:x_1-ix_2]$ takes the sphere to a complex quadric which satisfies the equation $z_0z_3=z_1z_2$ and is the image of the Segre embedding. This is a bijection on the complexification of the sphere; when the bijection is restricted to the real part, the coordinates of the image also satisfy the conditions that $z_1$ and $z_2$ are real, and that $z_3=\bar{z_0}$. The defined in the previous paragraph will define a ruling. Observe that in this case the two families of rulings are conjugate to each other.  

There are also two singular quadrics: a pair of planes and a cone; the latter has only one singular point whereas the former has infinitely many. 

We now revise the concept of a bi-degree of a rational curve, but more specifically for real rational curves. 

\begin{lemma}
  \label{bidegree}
  Given a curve $f:\mathbb{RP}^1\to Q\subset\mathbb{RP}^3$, there is a map $\tilde{f}:\mathbb{RP}^1\to \mathbb{RP}^1\times \mathbb{RP}^1$ such that $\theta\circ \tilde{f}=f$ as is in the diagram:\\
  \begin{center}\begin{tikzpicture}[node distance=2cm, auto]
    \node (C) {$\mathbb{RP}^1$};
    \node (P) [below of=C] {$\mathbb{RP}^1\times\mathbb{RP}^1$};
    \node (A) [right of=P] {$Q$};
    \draw[->] (C) to node {$f$} (A);
    \draw[->, dashed] (C) to node [swap] {$\tilde{f}$} (P);
    \draw[->] (P) to node [swap] {$\theta$} (A);
  \end{tikzpicture}\end{center}
  \end{lemma}
\begin{proof}
  Let $[s:t]\to[p_0(s,t):p_1(s,t):p_2(s,t):p_3(s,t)]$ denote the parametrization of a degree $d$ curve that lies on the quadric defined by $X_0X_3-X_1X_2=0$. Therefore $p_0p_3=p_1p_2$ and so $p_0p_3$ and $p_1p_2$ have the same factors. Let $q_0$ be the factor of highest degree common to $p_0$ and $p_1$ (it must be real since imaginary roots occur in pairs), and denote its degree by $m$. Therefore $q_2=p_0/q_0$  and $q_3=p_1/q_0$ are degree $d-m$ real polynomials.

By definition, $q_2$ divides $p_0$, and therefore it divides $p_1p_2$. If it shares a common factor $r$ with $p_1$ then $q_0r$ is a factor common to $p_0$ and $p_1$, contradicting the fact that $q_0$ was the highest common factor. Therefore, $q_2$ is co-prime to $q_1$ and so it must divide $p_2$ and $q_1=p_2/q_2$ is a degree $d-(m-d)=m$ real polynomial. 

This way $p_0=q_0q_2$,  $p_1=q_0q_3$,  and $p_2=q_1q_2$. We only have to check that $p_3=q_1q_3$, but $q_1q_3=p_1p_2/(q_0q_2)=p_1p_2/p_0$. Since $p_0p_3=p_1p_2$, we get $p_3=q_1q_3$.

$q_0$ and $q_1$ have the same degree $m$; $q_2$ and $q_3$ have the same degree $d-m$. Therefore the map defined by coordinates $[q_0:q_1]$ and $[q_2:q_3]$ are well defined. 
\end{proof}

\begin{definition}
  Given a rational map $f:\mathbb{RP}^1\to Q\subset\mathbb{RP}^3$, consider its corresponding map $\tilde{f}:\mathbb{RP}^1\to \mathbb{RP}^1\times \mathbb{RP}^1$ as defined in the previous lemma. Fix an orientation of each $\mathbb{RP}^1$ in both the domain and co-domain and let  $\pi_1$ and $\pi_2$ denote the projections onto the first and second coordinates respectively. $\pi_i\circ\tilde{f}:\mathbb{RP}^1\to\mathbb{RP}^1$, for $i=1,2$, are maps with degrees $m=\mathrm{deg}(\pi_1\circ\tilde{f})$ and  $n=\mathrm{deg}(\pi_2\circ\tilde{f})$. Then the \textbf{real bi-degree} of the curve is defined to be $(m,n)$. 
\end{definition}

The maps in lemma~\ref{bidegree} can be extended to the complexification. We can therefore define the complex bi-degree of the curve as follows:
\begin{definition}
  Given a rational map $f:\mathbb{CP}^1\to Q\subset\mathbb{CP}^3$, consider its corresponding map $\tilde{f}:\mathbb{CP}^1\to \mathbb{CP}^1\times \mathbb{CP}^1$ as defined in the previous lemma. If $m=\mathrm{deg}(\pi_1\circ\tilde{f})$ and  $n=\mathrm{deg}(\pi_2\circ\tilde{f})$, Then the \textbf{complex bi-degree} of the curve is defined to be $(m,n)$. 
\end{definition}

\begin{remark}
  The complex bi-degree of the curve as defined above is the same as the degrees of the homogeneous polynomials defining the map $\tilde{f}$. This means that if the complex bi-degree of the curve $f$ is $(m,n)$, then $m+n=\mathrm{the\ degree\ of\ the\ rational\ map\ } f$. Note that $m$ or $n$ may be negative when $(m,n)$ denotes the real bi-degree but not when it denotes the complex bi-degree.
\end{remark}

\begin{lemma}
  A curve $C$ in a quadric surface $Q$ cannot have bi-degree $(2,2)$.
  \label{bidegree22}
\end{lemma}
\begin{proof}This easily follows from the adjunction formula~\cite{hartshorne} $g=mn-m-n+1$ for a  genus $g$ curve of bi-degree $(m,n)$ embedded in $\mathbb{CP}^1\times\mathbb{CP}^1$. Curves of bi-degree $(2,2)$ would have to have genus $1$ and therefore cannot be rational. However, in this special case, it may also be seen by the following elementary means.
  
  The stereographic projection from a point on $Q$ which is not on the curve,  will project $C$ to a planar rational curve of degree~$4$. 

  Note that the tangent plane to the point of projection intersects the quadric in a pair of lines. By the definition of the bi-degree, the curve intersects each  of these lines in a pair of points. Each pair projects to the same point at the plane at infinity, giving rise to two nodes (one for each line) at the line at infinity in $\mathbb{RP}^2$. So it projects to a degree~$4$ rational curve with only two nodes, which is impossible because a rational curve with only nodes as its singularity has to have exactly $\frac{(d-1)(d-2)}{2}$ nodes. \end{proof}

\begin{lemma}
  A degree~$4$ real rational curve $C$ cannot lie in a real sphere.
\end{lemma}

\begin{proof}
  The complexification of the sphere is ruled by two families of complex lines. The two families are conjugate to each other. If the complexification $\mathbb{C}C$ of a real curve $C$ intersects a line from one of these families in $m$ points, then it must also intersect the conjugate of that line in $m$ points. This means that it would have to have bi-degree $(m,m)$. This is impossible for degree~$4$ real rational curves which cannot have bi-degree $(2,2)$ because of lemma~\ref{bidegree22}.
\end{proof}

We now note two lemmas that will help us when classifying degree~$4$ curves by reducing the number of real bi-degree cases to be checked to merely two cases.  

\begin{lemma}
  Consider a curve of real bi-degree $(m,n)$ in a curve in the hyperboloid $Q$  defined by $x_0x_3=x_1x_2$. Reflection along the plane $x_1=x_2$ preserves the hyperboloid, but interchanges the family of rulings and therefore changes the curve to one of bi-degree $(n,m)$. 
\end{lemma}
\begin{proof}
  When the Segre embedding $\theta:\mathbb{RP}^1\times \mathbb{RP}^1\to Q\subset \mathbb{RP}^3$ is composed with the map $r:\mathbb{RP}^1\times \mathbb{RP}^1\to \mathbb{RP}^1\times \mathbb{RP}^1$ which interchanges coordinates, it reverses the families. Explicitly: $([x_0:x_1], [x_2:x_3])\to ([x_2:x_3], [x_0:x_1])\to [x_0x_2:x_1x_2:x_0x_3:x_1x_3]$. But this is the same as the embedding $([x_0:x_1], [x_2:x_3])\to  [x_0x_2:x_0x_3:x_1x_2:x_1x_3]$ composed by a reflection in the plane given by $x_1=x_2$.
\end{proof}

\begin{lemma}
  Consider a curve $C$ of real bi-degree $(m,n)$ in the hyperboloid $Q$ defined by $x_0x_3=x_1x_2$. Reversing the orientation of one coordinate, results in a curve $\bar{C}$ of real bi-degree $(-m,n)$, which is projectively equivalent to $C$ by a transformation which is in the component of $\mathbb{P}GL_4(\mathbb{R})$ containing the identity. Therefore $\bar{C}$ is rigidly isotopic to $C$. 
\end{lemma}
\begin{proof}
  Reversing the orientation of one coordinate is equivalent to composing the Segre embedding $\theta:\mathbb{RP}^1\times \mathbb{RP}^1\to Q\subset \mathbb{RP}^3$ with the map $r:\mathbb{RP}^1\times \mathbb{RP}^1\to \mathbb{RP}^1\times \mathbb{RP}^1$, given  explicitly by $([x_0:x_1], [x_2:x_3])\to ([x_0:-x_1], [x_2:x_3])\to [x_0x_2:-x_1x_2:x_0x_3:-x_1x_3]$ which is the same as the embedding $([x_0:x_1], [x_2:x_3])\to  [x_0x_2:x_0x_3:x_1x_2:x_1x_3]$ composed by the projective transformation $[x_0:x_1:x_2:x_3]\to[x_0:-x_1:x_2:-x_3]$. This projective transformation is in the same component of $\mathbb{P}GL_4(\mathbb{R})$ that contains the identity. 
\end{proof}

\begin{lemma}
  All degree~$4$ rational curves in the hyperboloid are rigidly isotopic to either one of these curves, or their reflections:
\[[s(s^3+t^3):s(s^3-t^3):t(s^3+t^3):t(s^3-t^3)]\]
\[[s^2(s^2-4t^2):st(s^2-t^2):st(s^2-4t^2):t^2(s^2-t^2)]\]
\label{class4}
\end{lemma}

\begin{proof}We know that the curve must have bi-degree $(1,3)$ or $(3,1)$. If it has bi-degree $(3,1)$, consider its reflection with bi-degree $(1,3)$. After a change of coordinates, its parametrization $k:\mathbb{RP}^1\times \mathbb{RP}^1$ may be explicitly defined as:
  \[[s:t]\to ([s:t], [p_0:p_1])\]
  where each $p_i$ has degree 3. In that case $k$ is always an embedding, no matter what $p_i
  $'s are as long as they are of degree 3, since the linear part has no double points. Therefore, to ensure that a homotopy is an isotopy, we only need to avoid a common factor in the coordinate polynomials defining the second coordinate of $k$. As long as this is ensured, no singularities can result owing to the linear part. Therefore moving a root of $p_0$ will define a rigid isotopy as long it does not meet with a root of $p_1$. Similarly, moving a root of $p_1$ will define a rigid isotopy as long it does not meet with a root of $p_0$.

  There are three cases to be examined:

  \textit{Case 1 ($p_0$ and $p_1$ have imaginary roots):} Then the polynomials are each characterized by a pair of conjugate imaginary roots and therefore by an element of $\mathbb{CP}^1\setminus\mathbb{RP}^1$. It is easy to shift the complex roots of $p_0$ to $[1:\omega]$ and $[1:-\omega^2]$ and the roots of $p_1$ to $[l:-\omega]$ and $[1:\omega^2]$. Thereafter, there is no obstruction to shifting the real roots of $p_0$ and $p_1$ to $[1:-1]$ and $[1:1]$, thereby showing that all curves of this bi-degree form one isotopy class.
  \[[s:t]\to ([s:t], [s^3+t^3:s^3-t^3])\]

  \textit{Case 2 (either $p_0$ or $p_1$ has imaginary roots):} Assume that $p_0$ has imaginary roots,   and $p_1$ has all real roots, then the only real root of $p_0$ divides $\mathbb{RP}^1$ into only two components. At least two of the roots of $p_1$ are forced to lie in the same component and therefore may be combined to form a double root and then converted to an imaginary pair, reducing it to the above case. 

  \textit{Case 3 ($p_0$ and $p_1$ have only real roots):} The roots of $p_0$ divides $\mathbb{RP}^1$ into three components. If at least two of the roots of $p_1$ lie in the same component, then this pair can be combined to form an imaginary pair, thereby reducing this to the above case. 

  So finally the only case left is where both polynomials have real roots, and none of the roots of $p_1$ lie in the same component of $\mathbb{RP}^1\setminus \{\mathrm{roots\ of\ }p_0\}$. We may move the roots so that the roots of $p_0$ are $[1,-2], [0:1], [1,2]$ while the roots of $p_1$ are  $[1,-1], [1:0], [1,1]$, giving rise the following map: 

  \[[s:t]\to ([s:t], [s(s^2-4t^2):t(s^2-t^2)])\]

  Composing each of the above four curves on $\mathbb{RP}^1\times\mathbb{RP}^1$ with the Segre map $([x_0:x_1], [x_2:x_3])\to [x_0x_2:x_0x_3:x_1x_2:x_1x_3]$  gives a parametrization for each of the four isotopy classes of degree~$4$ up to mirror reflection.

They are as follows:
\[[s(s^3+t^3):s(s^3-t^3):t(s^3+t^3):t(s^3-t^3)]\]
\[[s^2(s^2-4t^2):st(s^2-t^2):st(s^2-4t^2):t^2(s^2-t^2)]\]
\end{proof}

When the stereographic projection is applied to a knot of degree $6$ in $S^3$ with one double point, it is sent to a real rational knot in $\mathbb{RP}^3$, which intersects the plane at infinity in two real points if the double point was real and only imaginary points if the double point was solitary.

A non-singular knot in $\mathbb{RP}^3$ may lie on a cone (if it lies on a plane it would have to be singular). However, the following lemma reduces its study to only those knots which lie on a hyperboloid. 

\begin{lemma}
  Consider a pair $(C,X)\in \mathcal{P}_{4,k}$ of a real rational knot $C$ and a plane $X$ so that the $\mathbb{C}C\cap\mathbb{C}X$ has four distinct points. $C$ can be perturbed by an arbitrary small perturbation to lie on a hyperboloid and still intersect the plane $\mathbb{C}X$ in  four distinct points and the same number of real points. 
\end{lemma}

\begin{proof}
  If $C$ lies on a cone, use a projective transformation to transform the cone to one defined by $x_0x_2-x_1^2=0$. So a parametrization $[p_0:p_1:p_2:p_3]$ of a curve on this cone is $p_0p_2=p_1^2$. This means that the roots of $p_1$ determine the roots of $p_0$ and $p_2$ while $p_3$ can be arbitrary; therefore the dimension of the curves on this cone are $10$. Since the space of cones is of dimension $8$, the total dimension of the space of curves lying on a cone is $18$ and therefore of codimension one in the space of all curves of degree~$4$ in $\mathbb{RP}^3$.

  So by small perturbation, the curve can be made to lie on a hyperboloid. The perturbation may be chosen small enough to not change the real intersection and ensure that the points in the intersection of the complexification are distinct.
\end{proof}
We now use our study of knots of degree~$4$ in $\mathbb{RP}^3$ to study knots of degree~$6$ in the sphere with one double point. 

\subsection{Degree~$6$ knots in $S^3$ with a real double point}
The real part of a real curve in $\mathbb{RP}^3$ is called affine if there is a real plane which is disjoint from it. We first prove the following lemma that will be used later:
\begin{lemma}
  A real curve of real bi-degree $(1,3)$ on a hyperboloid in $\mathbb{RP}^3$ cannot be affine. 
  \label{affinebidegree}
\end{lemma}
\begin{proof}
  Any plane $X$ intersects $Q$ in either an ellipse or a distinct pair of lines. In either case, the bi-degree of $\mathbb{R}Q\cap \mathbb{R}X$ is $(\pm 1, \pm 1)$ and it represents the homology class $\pm [L_1]+\pm [L_2]$, where $[L_1]$ and $[L_2]$ are the generators of $\mathrm{H}_1(Q)$ and are each represented by a line in the ruling. The homology class represented by the curve of bi-degree $(1,3)$ is $[L_1]+3[L_2]$, so the intersection number $([L_1]+3[L_2])\circ(\pm [L_1]+\pm [L_2])=\pm [L_1]\circ[L_2]\pm3[L_1]\circ[l_2]$, so the intersection number is either $\pm 2$ or $\pm 4$. 
\end{proof}

\begin{lemma}
  Consider a rational curve $C$ of real bi-degree~$(1,3)$ in a hyperboloid $Q$. All real lines $l$ such that $\mathbb{C}l$ intersects $\mathbb{C}C$ in two (conjugate) imaginary points lie in the same component of $\mathbb{RP}^3\setminus\mathbb{R}Q$.
  \label{onecomponent}
\end{lemma}
\begin{proof}
  Consider two lines $l_1$ (resp. $l_2$), such that $\mathbb{C}l_1$ (resp. $\mathbb{C}l_2$) intersect $\mathbb{C}C$  (paramaterized by $k$) in the points $k(z_1)$ and $k(\bar{z_1})$ (resp. $k(z_1)$ and $k(\bar{z_1})$). We may assume that $z_1$ and $z_2$ lie in the same component of $\mathbb{CP}^1\setminus \mathbb{RP}^1$, because if not, we may replace $z_2$ by its conjugate. So there is a path $z_t$ of imaginary points in one half of  $\mathbb{CP}^1\setminus \mathbb{RP}^1$, connecting $z_1$ and $z_2$. The conjugate $\bar{z_t}$ of the path connects $\bar{z_1}$ and $\bar{z_2}$. If a line $l_t$, connecting $z_t$ and $\bar{z_t}$, intersects $\mathbb{C}Q$ in more than two points, it would have to lie on $\mathbb{C}Q$ and would be one of the rulings. This is impossible since by definition, $l_t$ intersects the curve in at least two imaginary points but the curve has real bi-degree~$(1,3)$. Therefore the line $\mathbb{R}l_t$ never intersects $\mathbb{R}Q$ and therefore $\mathbb{R}l_1$ and $\mathbb{R}l_2$ must lie in the same component. 
\end{proof}

\begin{lemma}
  Consider a line $l$ intersecting a non-singular real rational curve $C$ of bi-degree $(1,3)$ in a hyperboloid $Q$ in only a pair of conjugate imaginary points. It defines a pencil of planes containing $l$. Any plane from that pencil intersects $C$ in two more \emph{distinct} real points.  \label{nontangent}
\end{lemma}

\begin{proof}
  The pencil may be parametrized by $\mathbb{RP}^1$ and each plane may be denoted by $X_t$, for some $t\in \mathbb{RP}^1$. 
 
  Since the knot is not affine (lemma~\ref{affinebidegree}), $\mathbb{R}X_t$ intsersects $\mathbb{R}C$. $X_t$ already intersects the $C$ in two imaginary points, so there are exactly two real points of intersection. We will now show that these points have to be distinct. 
 
 If a plane, say $X_0$, intersects $C$ in a real point with multiplicity two then for a small $\epsilon$, the intersections $\mathbb{C}X_t\cap \mathbb{C}C$ for all $t\in (-\epsilon,\epsilon)$ cannot be real because then the intersection of the planes with the curve would be tracing out local branches of the curve that are intersecting. We know that this curve does not have crossings. But then, since the knot is affine, $\mathbb{C}X_t\cap \mathbb{C}C$ cannot be purely imaginary for any $t$.
\end{proof}

\begin{lemma}
  Consider a homotopy of rational maps $q_t:\mathbb{RP}^1\to\mathbb{RP}^1$ of degree $3$. If there is a point $w_0$ such that $q_0(w_0)=z_0$ such that $\#(q_0^{-1}(z_0))=1$, then for each $t$ there is a point $w_t$ such that $q_t(w_t)=z_t$ such that $\#(q_t^{-1}(z_t))=1$.
  \label{interval}
\end{lemma}
\begin{proof}
  Consider the largest interval $I_0$ containing $w_0$ such that $q_0\restriction_{I_t}$ is injective. This interval must remain non-empty throughout the isotopy because for the smallest $t_0$ for which $I_{t_0}=\emptyset$, the end-points of the interval would merge to give point $z_{t_0}$ such that $\#(q_{t_0}^{-1}(q_{t_0}(z_{t_0})))=4$ which is impossible for a rational map of degree $3$. So at each stage we may continuously choose a point $z_t$ from $I_t$.
\end{proof}

\begin{lemma}
  Consider two pairs $(C_0,X_0)$ and $(C_1,X_1)$  in $\mathcal{P}_{4,2}$. If $C_0$ and $C_1$ are rigidly isotopic, then there is there is a path in $\mathcal{P}_{4,2}$ joining $(C_0,X_0)$ and $(C_1,X_1)$.
\end{lemma}

\begin{proof}
  We will solve these individually for the two types of real bi-degrees. Denote the path of rational knots connecting $C_0$ and $C_1$ by $C_t$, where $0\leq t\leq 1$. Let $k_t$ denote a parametrization of $C_t$, for each $t$.  Denote the intersection  by $\mathbb{C}X_0\cap\mathbb{C}C_0=\{k_0(z), k_0(\bar{z}), k_0(\lambda_1), k_0(\lambda_2)\}$, for some $z$ which is imaginary and some $\lambda_i$ which are real. Keep $z$ fixed throughout the isotopy and define the line $l_t$  to be the line defined by the two imaginary points $k_t(z)$ and $k_t(\bar{z})$. The points $k_t(z)$ and $k_t(\bar{z})$ must be distinct, otherwise it would give rise to a solitary double point. Now consider the two bi-degree cases separately:

  \textit{Real bi-degree $(1,3)$:} For each $t$, $l_t$ does not intersect $Q$ in any more points, otherwise it would be forced to lie on $Q$ and would have to be one of the rulings that intersect the curve in one real and two imaginary points. That is impossible because this curve is of \textit{real} bi-degree $(1,3)$. Therefore, by lemma~\ref{nontangent}, for each $t$, the unique plane $X_t$ defined by the line $l_t$ and $k_t(\lambda_1)$ must intersect the curve in two distinct real points. 

  Using lemma~\ref{onecomponent}, one can find a path of lines $l_t$ between the line $l_0$ and one which is defined by the conjugate pair in the intersection $\mathbb{C}X_1\cap\mathbb{C}C_1$. Again by lemma~\ref{nontangent}, the plane $X_t$ defined by each line $l_t$ and the fixed real point $k_t(\lambda_1)$ will intersect it in distinct points. Finally move $\lambda_1$ to a point which is in the preimage of the real point of intersection $\mathbb{R}X_1\cap\mathbb{R}C_1$. Once again lemma~\ref{nontangent} ensures that the curve and the plane will intersect in \textit{distinct} real points. 

  \textit{Real bi-degree $(1,1)$:}  
  Each plane $X_t$ containing $l_t$ will intersect the curve in the two real points $k_t(\lambda_t)$ and $k_t(\lambda_t')$. Since the curve is of degree~$4$, the point $\lambda_t$ can always be continuously chosen so the plane $X_t$ defined by it and the line $l_t$ intersects the curve in \textit{distinct} points. This can be seen as follows.

  Parametrize the pencil of planes containing a line $l$ by $\mathbb{RP}^1$. For each $s\in\mathbb{RP}^1$, denote the plane by $X_s$. If there is a plane $X_{s_0}$ that intersects the knot in a one double point, then for every $s$ small interval around $s_0$, it is impossible that $X_s$ intersects $k$ in only real points because the two real points of intersection will trace out two branches of the curve that intersect at the point where they coincide. This means that the double point must arise because $X_{s_0}$ is tangent to the knot. Supposing $X_{s_1}$ is another plane tangential to the knot, then let $I$ denote the interval component of $\mathbb{RP}^1\setminus\{s_0,s_1\}$ for which $X_s$ intersects the knot in real points for all $s\in I$. In that case, for each $s\in I$, $X_s$ is intersecting the curve in distinct real points that coincide at the for $X_{s_0}$ and $X_{s_1}$. The two real intersection points trace out a closed loop. Since the real part of a rational curve is connected, this loop must be the entire real curve.Since the rational curve is infinite, this interval must also be infinite, and therefore throughout the isotopy of lines, one may continuously choose a plane that intersects the knot in \textit{distinct} real points. 

  Consider the smallest $t_0$ for which the line $l_{t_0}$ lies on $Q$. $X_{t_0}$ will then be tangent to $Q$ and would intersect $Q$ in a pair of lines, one of which intersects the curve in one real point, and another which intersects the curve in a real point and a conjugate pair.
  
  The pullback of the line $l_{t_0}$ by the Segre embedding may be explicitly given by the set $([x_0:x_1], [\alpha_{t_0}:\beta_{t_0}])$ where $[\alpha_{t_0}:\beta_{t_0}]$ is fixed. Its points of intersections with the curve are given by those points $[p_0(s_i,t_i):p_1(s_i,t_i)]$ ($i=1,\ldots, 3$) such that $[p_2(s_i,t_i):p_3(s_i,t_i)]=[\alpha_{t_0}:\beta_{t_0}]$. Since the line intersects in two imaginary points and one real point, two of the $[s_i:t_i]$ are conjugate imaginary. By lemma~\ref{interval}, one can extend $[\alpha_{t_0},\beta_{t_0}]$ throughout the isotopy to $[\alpha_t,\beta_t]$ so that the line $l_t$ defined by $([x_0:x_1], [\alpha_{t}:\beta_{t}])$ intersects the curve in an imaginary pair and one real point. The line from the other family of rulings that passes through $[\alpha_t,\beta_t]$ intersects the curve in only one point, so together the two lines, one from each ruling, define a plane $X_t$ which intersects the curve in two real and two imaginary points. 
 
  By lemma~\ref{class4}, we know that $k_0$  is isotopic to the curve 
\[[s(s^3+t^3):s(s^3-t^3):t(s^3+t^3):t(s^3-t^3)]\] (or its mirror image)
  and by what we have just proved, the plane $X_0$ can be isotoped to a plane $X_0$ which intersects the curve $[s(s^3+t^3):s(s^3-t^3):t(s^3+t^3):t(s^3-t^3)]$ in two distinct real points throughout the isotopy. 

  By the same reasoning, $k_1$ can also be isotoped to the curve $[s(s^3+t^3):s(s^3-t^3):t(s^3+t^3):t(s^3-t^3)]$ and the plane $X_1$ to $X_1$ so that in intersects the curve it two distinct real points throughout the isotopy.  Observe that $X_0$ and $X_1$ are tangential to the ambient hyperboloid, but there the curve $[s(s^3+t^3):s(s^3-t^3):t(s^3+t^3):t(s^3-t^3)]$ intersects every line of its ruling in just one real point.  So every plane tangent to the hyperboloid intersects the curve in two real points, which are distinct as long as the point of tangency is not the curve. But the the complement of the curve in the hyperboloid is conencted and so $X_0$ may be moved to $X_1$ may be moved to each other via planes that intersect the curve in two distinct real points. 
\end{proof}

A degree four knot in $\mathbb{RP}^3$ which intersects the plane at infinity in two real points is either the unknot or the two crossing knot. Therefore, by the previous theorem, 

\begin{theorem}Singular knots in $S^3$ with one double point form a wall between two unknots or a wall between the unknot and the trefoil. \end{theorem}

\subsection{Degree~$6$ knots in $S^3$ with a solitary double point} 
Given a knot, if there is a plane disjoint from it, then it intersects the complexification of the knot in two pairs of conjugate imaginary points. Each pair defines a real line. Pick one of these real lines. It does not intersect the real part of the curve and indeed it will be shown that that such a line can always be chosen along an isotopy. We will need the following lemma. 

\begin{lemma}
  Consider a homotopy of rational maps $f_t:\mathbb{CP}^1\to\mathbb{CP}^1$, each of degree~$3$. For a fixed open set $U\subset\mathbb{CP}^1$, consider an $z_0\in f_t^{-1}(U)$. It is possible to continuously extend $z_0$ throughout the isotopy, so that for each $t$, $x_t\in f_t^{-1}(U)$
\end{lemma}
\begin{proof}
  Fix $w=f(z_0)$. The set of all rational functions $f$, such that $\#(f^{-1}(w))<3$ is a hypersurface $\mathcal{D}$ of co-dimension 1 in the space of rational maps of degree~3. The homotopy is a path in the space of  rational maps of degree~3, so we may perturb it to a path $t\to f_t$ which intersects $\mathcal{D}$ in finitely many points.  Let the points for which it intersects $\mathcal{D}$ be $t_1, t_2,\ldots, t_n$. For each $0\leq t\leq t_1$, we can always continuously chose a point from the three distinct points in $f_t^{-1}(w)$. We can then continue it for the next interval $t_1\leq t\leq t_2$ etc. 
\end{proof}

\begin{lemma} Consider a pair $(C,l)$ of a degree~$4$ knot and a line, such that $\mathbb{C}l$ intersects $\mathbb{C}C$ in a conjugate pair of imaginary points but is disjoint from the real point of the knot. Suppose there is an isotopy $C_t$ of the knot, where $C_0=C$, then this isotopy can be extended to a pair $(C_t,l_t)$, where for each $t$, $\mathbb{C}l_t$ intersects $\mathbb{C}C_t$ in a pair of conjugate imaginary points but no real points.  \end{lemma}
    \begin{proof}
      A degree~$4$ knot lies on a hyperboloid. The line that is defined, already intersects the knot and therefore the hyperboloid in two imaginary points. If it is not already transversal to the to the hyperboloid, perturb it so that it is. It cannot intersect the hyperboloid (and therefore not even the knot lying on it) in any more points and therefore is disjoint from the real part of $C$. 

      Let $k_t$ be the parametrization of $C$. We need to find a path $z_t$ so that and the line $l_t$ be defined by the points $k(z_t)$ and $k(\bar{z_t})$ does not intersect the curve in any more points. This means that the line $l_t$ must never be the real ruling. In other words, $[p_2(z):p_3(z)]$ must never be real in the parametrization $([s:t],[p_2:p_3])$, where the $p_i$'s are of degree~$3$, because the ruling is defined by $([x_0:x_1],[\alpha:\beta])$ for some real $\alpha$ and $\beta$ and it intersects the curve when for those values of $z$ such that $[p_2(z):p_3(z)]=[\alpha:\beta]$. But the previous lemma shows that this is always possible. 
\end{proof}

      If the line defined by a pair of conjugate imaginary points, $k(z)$ and $k(\bar{z})$, is such that no real point of the curve lies on it, then each real point and this line defines a plane belonging to the pencil of curves that is parametrized by $\mathbb{RP}^1$. This defines a map $\theta:\mathbb{RP}^1 \to \mathbb{RP}^1$.

\begin{lemma}If $\theta$ is surjective, it is a two sheeted covering of degree 2.\end{lemma}
  \begin{proof}  The image of a point by $\theta$ corresponds to a plane from the pencil that passes through the point. Since the degree is 4 it intersects one more point which, will turn out to be distinct. If the points coincide, then since the map is surjective, the preimage of a small enough neighbourhood of the pencil will be two intervals that intersect each other at one point. This is impossible since the knot does not have real double points.

Since each plane from the pencil intersects the curve in two real points, the preimage of a small enough interval around the point will be a disjoint union of two intervals and will therefore be evenly covered. \end{proof}

\begin{corollary}	Consider two isotopic knots $C_1$ and $C_2$. If there is a plane $X_1$ disjoint from $\mathbb{R}C_1$, then an isotopy of knots $C_t$ between $C_1$ and $C_2$, can be extended to a pair of ($C_t$,$X_t$), where $X_t$ is disjoint from $\mathbb{R}C_t$. \end{corollary}
  \begin{proof}  If there does not exist a plane disjoint from a knot $k$, then the map $\theta$ corresponding to it must be surjective and hence a two sheeted covering of degree 2. The map corresponding to any knot isotopic to $C$ will also have degree 2 and will therefore be surjective. 
  
    The planes which do intersect the curve in real points form an interval in the pencil. The complement is also an interval, which we have just shown is non-empty throughout the isotopy and therefore it is possible to always continuously choose a plane from it which will not intersect the knot in any real points. 
  \end{proof}
    Although the plane that is obtained at the end of the isotopy is disjoint from $k_2$, it need not be the one that we want, but we also have:

    \begin{lemma}If there are two planes $X_1$ and $X_2$ that intersect a knot in only imaginary points, then they are connected by a path of planes $X_t$ that are all disjoint from the real part of the knot.\end{lemma}
      \begin{proof}      The two planes intersect in a line that define a pencil of planes. Again, we may define a map $\theta: \mathbb{RP}^1 \to \mathbb{RP}^1$ by taking each point on the real part of the knot to the plane defined by it and the line. The image is a connected interval that is a proper subset  of $\mathbb{RP}^1$ because the two given planes do not intersect the knot in real points. The complement of a connected interval in $\mathbb{RP}^1$ is an interval and therefore one can choose path to connect the planes in the pencil. \end{proof}

Observe that this wall is the same as one of the earlier walls because we can isotope the knot so that the solitary double point becomes a cusp and thereafter a non-solitary double point. This can be seen by considering its corresponding pair of the real degree~$4$ rational curve and the plane that intersects the curve in only imaginary points. Move the plane so that it is tangential to the curve and thereafter intersects it in a real pair. At the stage that it is tangential to the curve is when the curve lifts to one in the sphere that has a cusp. 

We have proved that:
\begin{theorem}
  In the space of rational curves of degree~$6$ in $S^3$, there are $4$ rigid isotopy classes of curves that have only one double point as a singularity. 
\end{theorem}

\subsection{The space of real rational curves in $S^3$}
\label{manifold}
The space of rational curves in $\mathbb{RP}^3$ is a projective space and therefore a closed manifold. When we consider those curves in $\mathbb{RP}^4$ whose image lies in $S^3$, the situation is no longer straightforward. However, we will prove that this space is also a closed manifold. This will enable us to use the fact that the top homology is $\mathbb{Z}/2$ and the Poincare duality to find a bound on the number of components. We will then show that in the case of degree~$6$, the bound obtained is sharp.

\subsubsection{The space of real rational curves in $S^3$ is a manifold}
We will need two lemmas from linear algebra to prove the following theorem which states that the space of degree~$6$ rational curves in $S^3$ is a manifold. 
\begin{lemma}
  \label{trivial}
  If $h_0, h_1, \dots , h_{d}$ are linearly independent polynomials of degree $\leq d$, then  the determinant of the following matrix will be non-zero:

  \[\left(\begin{array}{cccc}
      h_0(t_1) & h_1(t_1) &\cdots & h_{d}(t_1)\\
      h_0(t_2) & h_1(t_2) &\cdots & h_{d}(t_2)\\
      \vdots&\cdots&\cdots&\vdots\\
      h_0(t_{d+1}) & h_1(t_{d+1}) &\cdots & h_{d}(t_{d+1})\\
    \end{array}\right)\]
    as long as all the $t_i$'s are distinct.
\end{lemma}
\begin{proof}
  No linear combination of $\{h_0, h_1, \dots , h_{d}\}$  can result in the zero polynomial since they are linearly independent. However, if the columns of this matrix are linearly dependant, it is equivalent to a linear combination of $h_i$'s having $d+1$ roots, which is impossible for a polynomial of degree less than $d+1$.
\end{proof}
\begin{lemma}
  \label{resultant}
  Let $p_0(t)$ and $p_1(t)$ be two polynomials of degree~$d$ that do not share a root. Then the set  \[\{p_0(t), tp_0(t),\ldots, t^{d-1}p_0(t), p_1(t), tp_1(t),\ldots, t^{d-1}p_1(t),t^dp_1(t),\ldots,t^{d+k}p_1(t)\}\] is linearly independent in the space of polynomials of degree $2d+k$. 
\end{lemma}
\begin{proof}
  Let $p_0(t)=a_0+a_1t+\ldots+a_d$ and $p_1(t)=b_0+b_1t+\ldots+b_d$ where $a_d\neq 0 \neq b_d$.
  The set of polynomials \[\{p_0(t), tp_0(t),\ldots, t^{d-1}p_0(t), p_1(t), tp_1(t),\ldots, t^{d+k-1}p_1(t),t^{d+k}p_1(t)\}\] is linearly independent if and only if the following matrix is non-singular:
\[
  \left(\begin{array}{cccccccccccc}
  a_0 & a_1 & \cdots & a_{d-1} & a_d & 0   & 0   & 0      & 0   & 0 & &0 \\
  0   & a_0 & a_1 & \cdots & a_{d-1} & a_d & 0   & 0   & 0      & 0 & &0 \\
&&&&&\vdots&&&&& \\
0   & 0   & 0      & 0   & 0   & a_0 & a_1 & \cdots & a_{d-1} & a_d &  &0 \\
b_0 & b_1 & \cdots & b_{d-1} & b_d & 0   & 0   & 0      & 0   & 0 & \vdots &0 \\
0   & b_0 & b_1 & \cdots & b_{d-1} & b_d & 0   & 0   & 0      & 0 & \vdots &0 \\
&&&&&\vdots&&&&&& \\
0   & 0   & 0      & 0   & 0   & b_0 & b_1 & \cdots & b_{d-1} & b_d & &0 \\
&&&&&\vdots&&&&&& \\
0 &0   & 0   & 0      & 0   & 0   & b_0 & b_1 & \cdots & b_{d-1} & &b_d
\end{array}\right)
  \]

  Observe that the submatrix which is obtained by removing the last row and column is the Sylvester matrix whose determinant is the resultant $\mathrm{R}(p_0,p_1)$ which is non-zero because the $p_0$ and $p_1$ do not share a root. The determinant of the original matrix is $\pm b_d^{k_1}\mathrm{R}(p_0,p_1)$ which is non-zero because  $b_d\neq 0$.
\end{proof}
\begin{theorem}
  The space of real parametrizations of rational curves of degree $6$ in $\mathbb{RP}^4$ that lie in the sphere, is a manifold of dimension $21$ (co-dimension $13$).
\end{theorem}
\begin{proof}
   Denote any rational parametrization explicitly by $[p_0:p_1:p_2:p_3:p_4]$, where each $p_i$ is a homogeneous polynomial of the same degree. Let $q(t)=p_1(t)^2+p_2(t)^2+p_3(t)^2+p_4(t)^2-p_0(t)^2$. The space of all rational parametrizations in $\mathbb{RP}^4$ is the projective space $\mathbb{RP}^{34}$. We will restrict attention to each chart, defined by one coefficient of one of the polynomials, and dividing all the coordinate polynomials by that coefficient. We will prove that for each chart, we can find a submersion to $\mathbb{RP}^{13}$ whose pullback of $0$ will be parametrizations of curves that lie in $S^3$

   Fix points $t_1, t_2, \ldots, t_{13}$ in $\mathbb{RP}^1$. Define the map $\theta: \mathbb{R}^{34}\to \mathbb{R}^{13}$ that takes the parametrization $[p_0:p_1:p_2:p_3:p_4]$ (restricted to a chart) to $(q(t_1),q(t_2),\ldots,q(t_{13}))$. The inverse image of $(0,0,\ldots,0)\in \mathbb{R}^{13}$ is precisely the set that we want because any degree~$6$ rational curve the intersects the sphere in $13$ distinct points must lie on the sphere. 

   Represent a parametrization in $\mathbb{RP}^{34}$ by coordinates $[a_0:a_2:\ldots:a_{34}]$ where  $a_0, \ldots, a_6$ are the coefficients of $p_0$, $a_7, \ldots, a_{13}$ are the coefficients of $p_1$, and so on. Each $q(t_i)$ is a quadratic polynomial in $a_i$ for a fixed $t_i$. 

   It is enough to show that for any arbitrary point in this space, the differential $d\theta$ is surjective. $d\theta$ is a $13\times 34$ matrix where the rows are of the form $(\frac{\partial q(t_i)}{\partial a_j})_{(i,j)}$. By using the chain rule on each entry it is easy to see that if $a_i$ is the $l$th coefficient of $p_k$, the matrix is $(\pm2 p_k(t_j)t_j^l)_{(i,j)}$. The sign is $-$ when $k=6$. 

   We now have to find a minor of the above matrix which has a non-zero determinant. Choose two $p_{i_1}$ and $p_{i_2}$ which do not share a root and whose coefficients are not the with the coefficient defining the chart. Consider the following $13\times13$ minor:

\[\left(\begin{array}{cccccc}
    p_{i_1}(t_1) & \ldots & t_1^5p_{i_1}(t_1)&p_{i_2}(t_1) & \ldots & t_1^6p_{i_2}(t_1)\\
      \vdots&\cdots&\cdots&\cdots&
      \cdots&\vdots\\
      p_{i_1}(t_{13}) & \ldots & t_{13}^5p_{i_1}(t_{13})&p_{i_2}(t_{13}) & \ldots & t_{13}^6p_{i_2}(t_{13})\\
    \end{array}\right)\]

    Lemma~\ref{resultant} shows that the polynomials 
    $\{p_{i_1}(t) , \ldots , t^5p_{i_1}(t),p_{i_2}(t) , \ldots , t^6p_{i}(t)\}$ are linearly independant, and lemma~\ref{trivial} shows that the above matrix is non-singular. 

    If the point $[a_0:a_2:\ldots:a_{34}]$ corresponds to a parametrization $[p_0:p_1:p_2:p_3:p_4]$ where two coordinate polynomials share a root, it means that the curve intersects the co-dimension two linear subspace defined by two coordinates of $\mathbb{RP}^4$ equal zero. By the transversality theorem, it is possible to find a diffeomorphism (for instance, a rotation) that preserves $S^3$ and takes the curve to one which does not intersect any of the co-dimension $2$ hypersurfaces defined by $x_i=0$ and $x_j=0$. Simply compose this diffeomorphism with the map $\theta$ to define a new map which will be regular when restricted to a small enough neighbourhood of $[a_0:a_2:\ldots:a_{34}]$ in a chart containing it. 
\end{proof}

\begin{remark}
The above theorem easily generalizes to the curves of higher degrees.
\end{remark}
\subsubsection{The space of real rational curves in $S^3$ is connected}
We do this in two steps because it is easier to prove the theorem for singular curves:

\begin{lemma}
  The space of real rational curves of degree $6$ in $S^3$ with at least one double point, is connected.
\end{lemma} 
  \begin{proof}
    Given two curves in $S^3$ with at least one double point, use a linear transformation so that they have a common double point. Project each of these curves from their common double point. Denote the projections by $C_1$ and $C_2$ and their respective parametrizations by $k_1$ and $k_2$. They each intersect the plane at infinity in at least one conjugate pair say $k_i(z_i)$ and $k_i(\bar{z_i})$.
    
    We know that each rational curve of degree~$6$ with one double point in $S^3$ corresponds to a pair $(C,P)$ where $C$ is a degree~$4$ real rational curve  in $\mathbb{RP}^3$ and $P$ is a plane that intersects $k$ in $4$ points such that at least two of them form a conjugate imaginary pair. 
    
    Consider a path $k_t$ between $k_1$ and $k_2$ and a path $z_t$ between $z_1$ and $z_2$. Extend the path to the path of pairs $(k_t,P_t)$ where each $P_t$ is continuously chosen so that it contains $k_t(z_t)$ and $k_t(\bar{z_t})$. Now make the plane $P_t$ coincide with the original plane at infinity. At each stage we get a pair that will be pulled back to a degree~$6$ curve in the sphere. 
  \end{proof}
  \begin{lemma}
  Given a curve $k_1$ in the space of real rational curves in $S^3$, there is a curve with one double point in its path connected component. 
\end{lemma}
\begin{proof}
  Project the curve from a point on it to a degree~$5$ curve in $S^3$. The projection will be a curve which intersects the plane at infinity in one real point and two conjugate pairs of imaginary points that lie on the empty conic. Denote the parametrization of the curve by $k:\mathbb{RP}^1\to\mathbb{RP}^3$, and fix the imaginary pairs $k(z_i)=v_i$ and $k(\bar{z_i})=\bar{v_i}$ for some fixed $z_i$ where $i=1, 2$. 
  Fix two distinct points $\lambda_1$ and $\lambda_2$ in $\mathbb{RP}^1$, then we can always find a curve so that $k(\lambda_1)=w_1$ and  $k(\lambda_2)=w_2$ for any pair $w_1$ and $w_2$, while still maintaining $k(z_i)=v_i$ and $k(\bar{z_i})=\bar{v_i}$ so that it lies on the imaginary conic at infinity and the curve can be lifted back to $S^3$. Now given a path $w_t$ between $w_1$ to $w_2$, for each $w_t$ there is a $k_t$ such that $k_t(\lambda_1)=w_t=k(\lambda_2)$, $k_t(z_i)=v_i$ and $k_t(\bar{z_i})=\bar{v_i}$ and to create a double point. 
\end{proof}

By the previous two lemmas we obtain the theorem:
\begin{theorem}
  The space of real rational curves of degree~$6$ in the sphere is connected.  \qed
\end{theorem}

We now prove that:
\begin{lemma}
  The space of rational curves with \emph{exactly} one double point is a manifold of dimension $20$. 
  \label{mansing}
\end{lemma}
\begin{proof}
 Denote the space of rational curves with at least one double point as $D$ and the space of rational curves with more than one double point as $D'$, then 
 \[D\setminus D' \cong S^3 \times D_p\] where $D_p$ denotes the space of curves with exactly one double point at the north pole $p$. This equivalence map can be seen as follows: $S^3$ may be treated as the unit quarternions with the north pole $p$ as the identity quarternion. The quarternions act transitively on $S^3$ and extend linearly to $\mathbb{R}^4$. The map is then defined by sending a curve $k$ with one double point at $q$ to the pair $(q,k_0)$ where $k_0$ is the action of $A_q$ on $k$, where $A_q$ is the quarternion that takes $q$ to $p$. 

 So it is enough to show that $D_p$ is a manifold. Via the stereographic projection, this space is isomorphic to the space of knots of degree~$4$ in $\mathbb{RP}^3$ that intersects the plane at infinity in four points, two of which lie on the empty conic. The curves that intersect the plane at infinity in the points $[0:0:1:i]$ and $[0:0:1:-i]$ is a manifold of dimension $17$, which will be denoted by $K_i$. Fix a real point $\lambda_1$ on the plane at infinity and one point $\lambda_2$ outside it. Given any conjugate pair $z$ and $\bar{z}$ on the empty conic, there is unique transformation $T_z$ which fixes $\lambda_1$ and $\lambda_2$ and takes $z\to i$ and $\bar{z}\to -i$. For an open neighbourhood $U$ around $z$ in the empty conic, the map and a neigbhbourhood $V$ around a chosen knot, $T_z$ gives an isomorphism \[V\cong U\times K_i\] that is defined by taking a knot $k$ which intersects the empty conic in $z$ and $\bar{z}$ to $(z,T_z(k))$.
\end{proof}
  \subsubsection{Computing a bound on the number of components}

Let $X$ denote the space of parametrizations of real rational knots of degree $6$ in $S^3$. Its dimension is $21$. Denote by $D$, the co-dimension $1$ subspace of parametrizations which have at least one singularity.  

Since $X$ is a manifold, by Poincare duality:
\[\mathrm{H}^0(X\setminus D;\mathbb{Z}/2)=\mathrm{H}_{21}^{BM}(X\setminus D;\mathbb{Z}/2)\]
By the long exact sequence of Borel-Moore homology~\cite{borelmoore},

\begin{equation}
  \mathrm{dim\ }\mathrm{H}_{21}^{BM}(X\setminus D;\mathbb{Z}/2)\leq\mathrm{dim\ }\mathrm{H}_{21}^{BM}(X;\mathbb{Z}/2)+\mathrm{dim\ }\mathrm{H}_{20}^{BM}(D;\mathbb{Z}/2)
  \label{eqn1}
\end{equation}

$\mathrm{dim\ }\mathrm{H}_{21}^{BM}(X;\mathbb{Z}/2)= 1$ since $X$ is a connected manifold of dimension $21$ any manifold is $\mathbb{Z}/2$-orientable. Consider the space of curves which have only one double point or a cusp as a singularity. Denote the complement of this space in the space of curves with singularities to be $D'$; it is of co-dimension $2$. Then

  \begin{equation}
    \mathrm{dim\ }\mathrm{H}_{20}^{BM}(D;\mathbb{Z}/2)\leq\mathrm{dim\ }\mathrm{H}_{20}^{BM}(D';\mathbb{Z}/2)+\mathrm{dim\ }\mathrm{H}_{20}^{BM}(D\setminus D';\mathbb{Z}/2)
  \label{eqn2}
  \end{equation}

  Again by duality, since $D\setminus D'$ was shown to be a manifold by lemma~\ref{mansing}, \[\mathrm{H}_{20}^{BM}(D\setminus D';\mathbb{Z}/2)=H^0(D\setminus D';\mathbb{Z}/2)\] which we know has rank $4$, and $D'$ has co-dimension $2$ so $\mathrm{dim\ }\mathrm{H}_{20}^{BM}(D';\mathbb{Z}/2)=0$. Therefore, by the inequality~\ref{eqn1},
  \[H^0(X\setminus D;\mathbb{Z}/2)\leq1+4=5\]

We will now try and obtain representatives for these five components. 
\begin{figure}[t]
	\begin{center}
		\includegraphics[height=3cm]{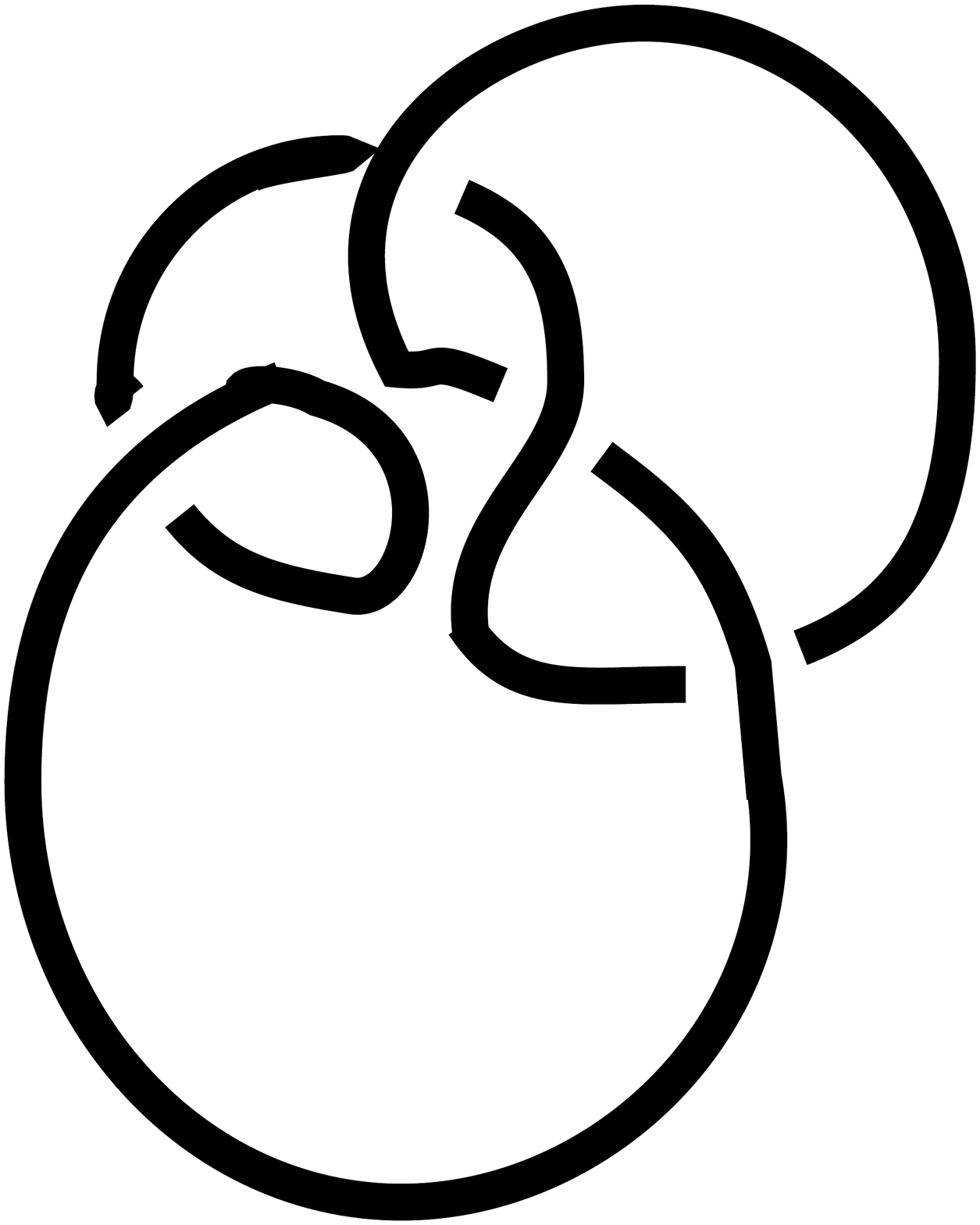}
		\includegraphics[height=3cm]{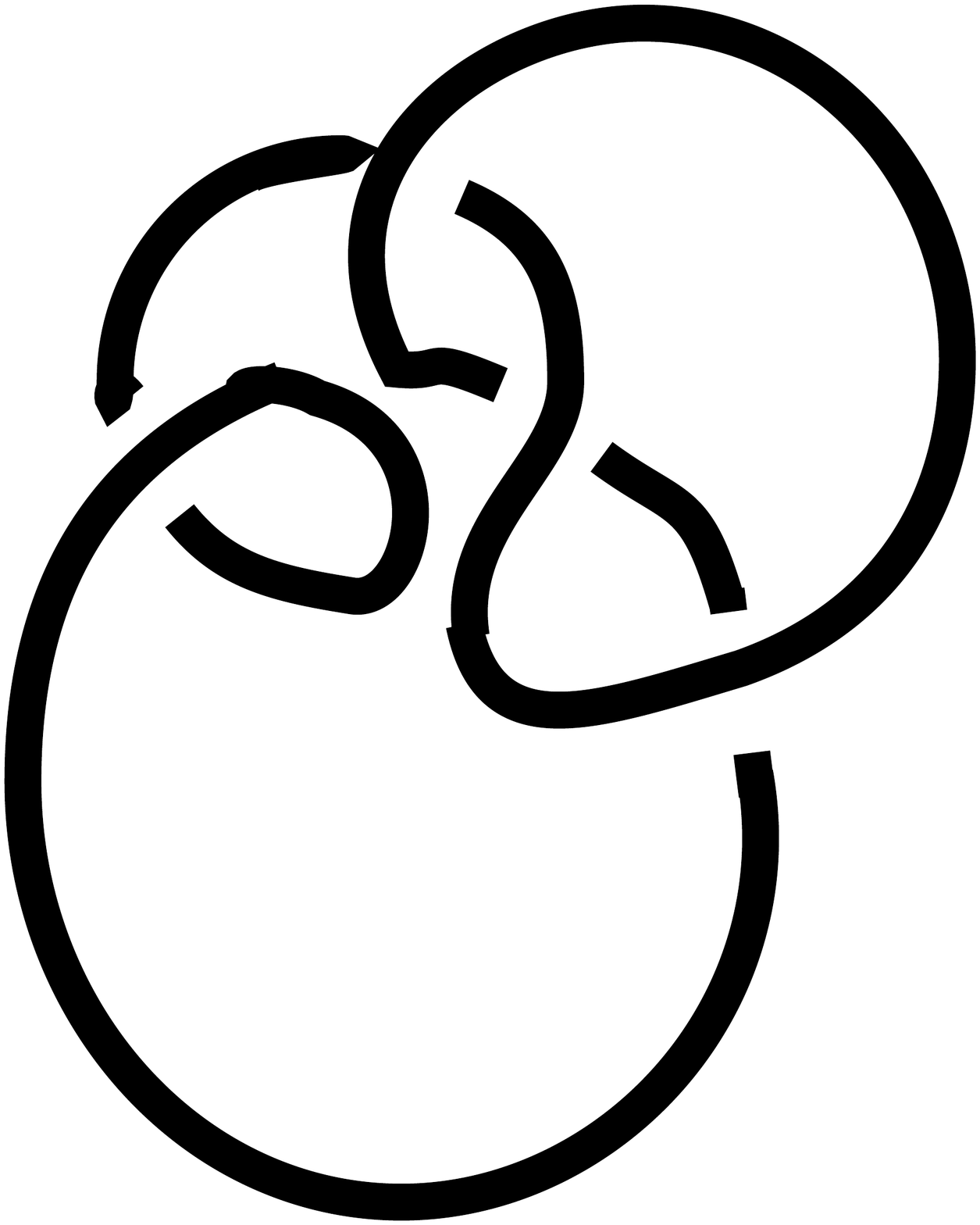}
		\includegraphics[height=3cm]{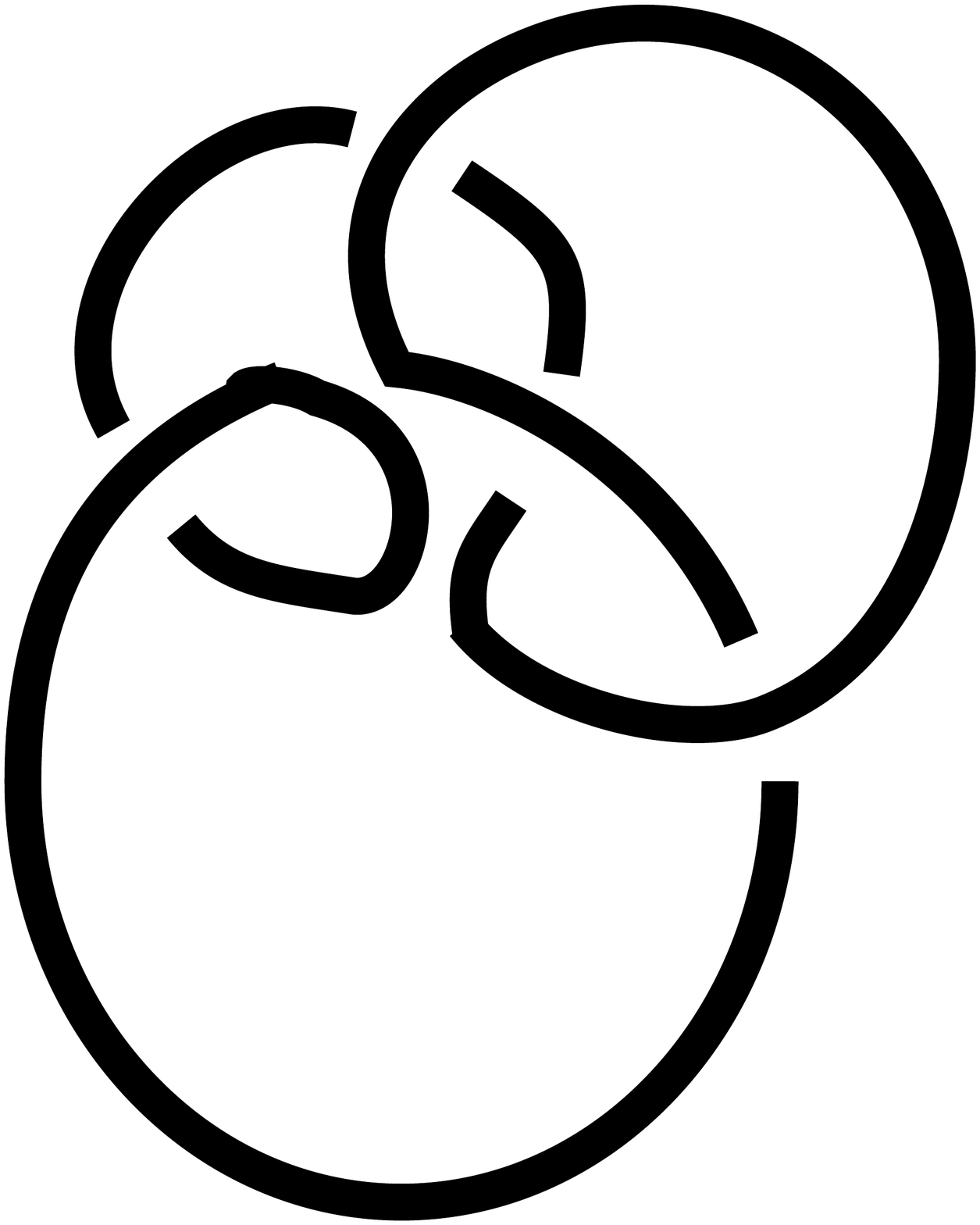}
		\includegraphics[height=3cm]{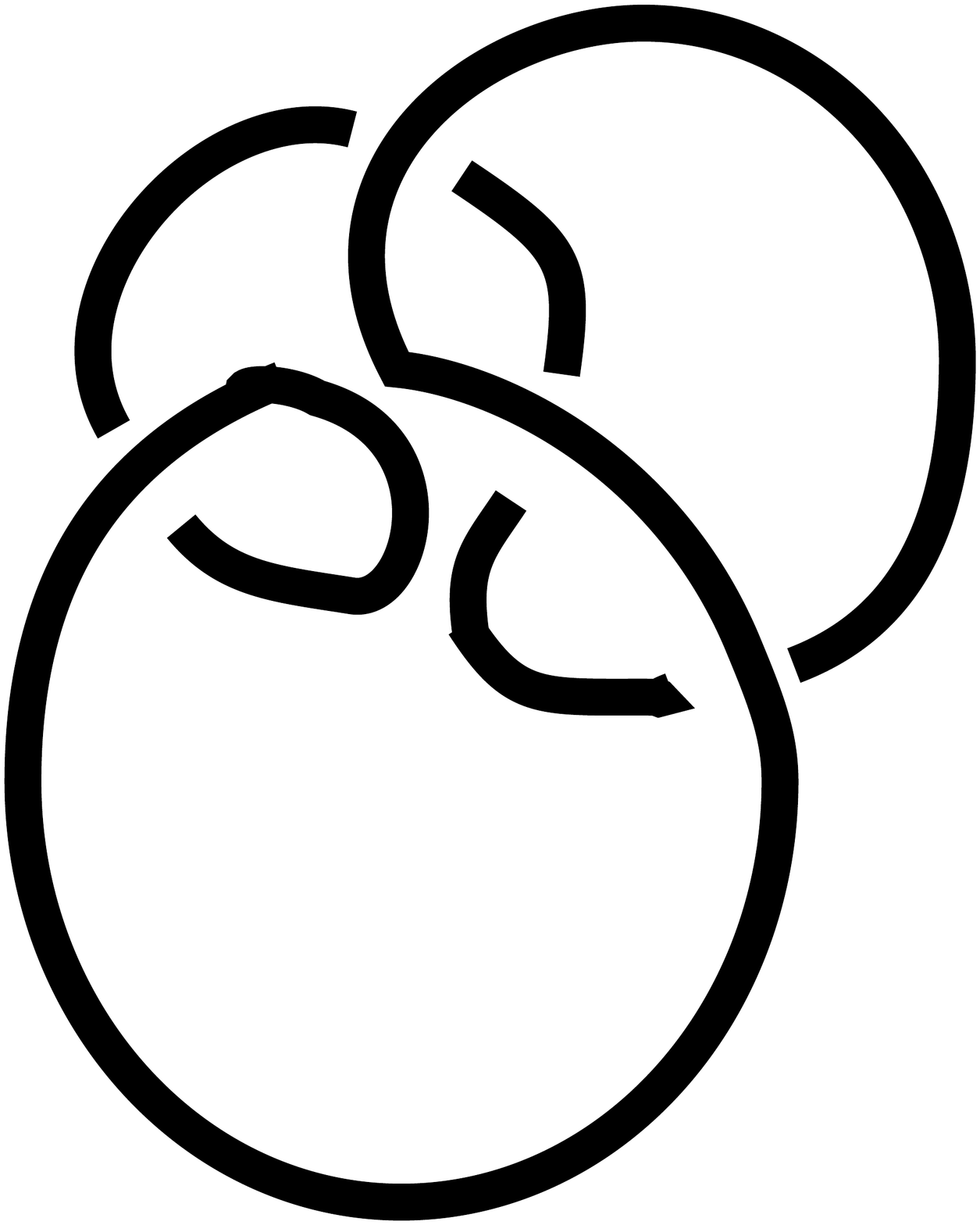}
	\caption{Knots of degree $6$ writhes 4, 2, 0, and 2. The first is topologically isotopic to the trefoil and the other two are topologically isotopic to the unknot. The successive perturbations are formed by a vertical rotation of the innermost circle.}
	\end{center}
	\label{fig:alldeg6}
\end{figure}

The first three knots shown in figure~\ref{fig:alldeg6} are not isotopic to each other because they have distinct writhe numbers: $4$, $2$, and $0$. There is a wall separating the knots with writhes $4$ and $2$ and this is the wall obtained by the pull back in the sphere of the two crossing knot of degree~$4$ in $\mathbb{RP}^3$.

Observe that the trefoil $T^+$ and its mirror image $T^-$ lie in separate components and therefore account for two of the rigid isotopy classes. They can be obtained by perturbing the pullback of the two-crossing knot. 

Now consider the pullback of the unknot. Perturbing it will lead to knots of different writhes and therefore one of them, denoted by $U^+_2$, has got to have a non-zero writhe. Its mirror image, denoted by $U^-_2$, would have a negative writhe number and therefore will lie in a separate component. Together, they account for two more components. 

Since we need to account for only one more component, the other knot formed by a pullback of the unknot, and denoted by $U_0$, is isotopic to its mirror reflection. It must have writhe number $0$ and therefore $U^+_2$, and $U^-_2$ must have writhes $+2$ and $-2$ respectively. 

%Figure~\ref{fig:trefoil6} shows that the trefoil has writhe number $\pm 2$. The method of construction will be explained in Section~\ref{examples}. Since we have accounted for the five isotopy classes, the other knot besides the trefoil which can be obtained by perturbing the pullback of the two crossing knot, must have writhe number distinct from it and therefore will be of writhe number $0$. So we have shown that degree~$6$ knots are rigidly isotopic to either:
\begin{enumerate}
  \item A knot of writhe number $2$ which is topologically isotopic to the trefoil (and its mirror image)
  \item A knot of writhe number $2$ which is topologically isotopic to the unknot (and its mirror image)
  \item A knot of writhe number $0$ which is topologically isotopic to the unknot.  It is rigidly isotopic to its mirror image.
\end{enumerate}
and that the two walls divide the knots of different writhes. 

\section{Degree 8}
\label{sec:deg8}
For real rational knots of degree~8 in the sphere, the rigid isotopy classification is more difficult. However, it is possible to connection between the isotopy classes of degree~8 real rational knots in $S^3$ with exactly one double point, and degree~6 real rational knots in $\mathbb{RP}^3$.

\begin{theorem}
    Any pair $(C,X)\in \mathcal{P}_{6,0}\cup\mathcal{P}_{6,2}$ is rigidly isotopic to a pair $(C',X,)\in\mathcal{P}'_{6,0}\cup\mathcal{P}'_{6,2}$ which can be lifted to the sphere. 
\end{theorem}
\begin{proof}
  $C$ intersects $X$ in at least four non-real points. The four points define a pencil of conics. According to the classification of real pencil of conics in the plane~\cite{degtyarev1996quadratic}, a base of four distinct non-real points (in conjugate pairs) defines a unique real pencil of conics. We know that there is a pencil of conics containing the empty conic which has four distinct non-real points as a base. Therefore we can find an empty conic containing the four non-real points. By a linear transformation, one can transform this conic to a standard empty conic. 
\end{proof}
\begin{corollary}
If a degree~6 real rational knot in $\mathbb{RP}^3$ intersects a plane in a maximum of two real points, it is rigidly isotopic to the stereographic projection of a real rational knot of degree~8 in the 3-sphere with one double point. 
\end{corollary}
\section{Constructing examples of knots in $S^3$}
Before demonstrating why Bj{\"o}rklund's~\cite{bjorklund} method of constructing examples applies to curves in the sphere, we note that following simple lemma which proves the existence of torus knots in rational curves of each degree $d$.
\begin{lemma}
  There exist $(m,d/2)$-torus knots in $S^3$ of degree~$d$ for each $m$ which is coprime to $d/2$.
\end{lemma}
\begin{proof}
  Treat the real 3-sphere as a subset of $\mathbb{C}^2$ defined by $|z_1|^2+|z_2|^2=1$. It is easy to see the that curve defined by $t\to (e^{mit},e^{dit/2})$ (where $m<d/2$) is a rational immersion of degree~$d$. It is an embedding if and only if $(m,d/2)$ are co-prime.
\end{proof}
\label{examples}
We will use the following restatement of the theorem proved in~\cite{bjorklund}: 

\begin{theorem}[Bj{\"o}rklund]
  Consider two rational knots $C_1$ and $C_2$ of degrees $m$ and $n$ in $\mathbb{RP}^3$ that intersect in only one point $p$ with linearly independent tangent lines at $p$. There exist parametrizations $[p_0:p_1:p_2:p_3]$ and $[q_0:q_1:q_2:q_3]$ of $C_1$ and $C_2$ respectively, so that $[p_0q_0:p_1q_0+q_1p_0:p_2q_0+q_2p_0:p_3q_0+q_3p_0]$ defines a parametrization of a knot of degree $m+n$ which is isotopic to a small perturbation of the union of $C_1$ and $C_2$ at $p$.
\end{theorem}

The stereographic projection allows us to obtain examples of knots in $S^3$, however we would also want to maintain the right degree. This is possible because of the following observation: that the perturbation does not affect the points at infinity. We therefore obtain the following couterpart for knots in $S^3$.

\begin{theorem}
  \label{construction}
 If there are two knots in the sphere with degrees $m$ and $n$ in $S^3$ which intersect in only one point $p$, then it is possible to construct a knot with a parametrization of degree $m+n$ in the sphere which is isotopic a small perturbation at $p$ of the union of the original knots outside a small neighbourhood of the intersection. 
\end{theorem}
\begin{proof}
  Consider the two knots $C_1$ and $C_2$ of degrees $m$ and $n$ in $S^3$ that intersect in only one point $p$ so that their the tangents at the point are linearly independent. Choose a point of projection $N$ that does not lie on either of the two curves and project them to a copy of $\mathbb{RP}^3$. By theorem~\ref{maincorrespondence}, $\pi_N(\mathbb{C}C_1)$ and $\pi_N(\mathbb{C}C_2)$ intersect the blow-up of $N$ in $m$  and $n$ points respectively, that lie on an empty conic.  Choose coordinates so that the blow-up of $N$ is the plane at infinity. 
  
  Now by theorem~\ref{construction}, we can choose parametrizations $[p_0:p_1:p_2:p_3]$ and $[q_0:q_1:q_2:q_3]$ of the curves $\pi_N(C_1)$ and $\pi_N(C_2)$ so that  $[p_0q_0:p_1q_0+q_1p_0:p_2q_0+q_2p_0:p_3q_0+q_3p_0]$  parametrizes a degree~$m+n$ knot $C$  which is isotopic to a small enough perturbation of the union of $\pi_N(C_1)$ and $\pi_N(C_2)$ at $p$. We will now show that $\pi^{-1}_N(C)$ is a knot \emph{of the same degree} $m+n$ inside $S^3$. 
  
  Let $[0:\alpha_1:\alpha_2:\alpha_3]$ be a point of intersection of $\pi_N(\mathbb{C}C_1)$ with the plane at infinity. It must lie on the empty conic. Let its preimage under the parametrization of $C_1$ be $[\delta:\gamma]$. In that case $p_0(\delta,\gamma)=0$.  Denote the image under the parametrization of $C_2$ of $[\delta:\gamma]$ to be the point $[\beta_0:\beta_1:\beta_2:\beta_3]$. In that case, the image of $[\delta:\gamma]$ under the parametrization of $C$ is  $[0\beta_0:\alpha_1\beta_0+q_10:\alpha_2\beta_0+\beta_20:\alpha_3\beta_0+\beta_30]$=$[0:\alpha_1\beta_0:\alpha_2\beta_0:\alpha_3\beta_0]$=$[0:\alpha_1:\alpha_2:\alpha_3]$.  So each of the $m$ intersection points of $\pi_N(\mathbb{C}C_1)$ with the plane at infinity is also the intersection point of $\mathbb{C}C$ with the plane at infinity, with the plane at infinity.  Similarly, each of the $n$ intersection points of $\pi_N(\mathbb{C}C_2)$ with the plane at infinity, is also the intersection point of $\mathbb{C}C$ with the plane at infinity. Therefore, each of the $m+n$ intersection points of $\mathbb{C}C$ with the plane of infinity must lie on the same empty conic. By theorem~\ref{construction} $\pi^{-1}_N(C)$ is a knot of degree $m+n$ inside $S^3$
\end{proof}

The following lemma will be useful in constructing many examples:
\begin{lemma}
  A degree $2$ curve $C$ in $S^3$ projects under the stereographic projection $\pi_p$ to a degree~$2$ curve $\pi_2(C)$ in $\mathbb{RP}^3$ if and only if $\pi_2(C)$ is a circle.
\end{lemma}
\begin{proof}
  $\pi_2(C)$ is a conic  that lies on a plane which can be transformed via an orthogonal transformation to the plane defined by $x_0=0$. The plane defined by $x_0=0$ intersects the empty conic at the plane at infinity in the conjugate pair $[0:0:1:i]$ and $[0:0:1:-i]$. $\pi_2(C)$ would be the image under the stereographic projection of a curve of the same degree if and only if its two points of intersection with the plane at infinity lie on the empty conic. Since its ambient plane intersects the conic in $[0:0:1:i]$ and $[0:0:1:-i]$, $\pi_2(C)$  would have to contain these two points. It is easy to see that conics that lie on the plane $x_0=0$ and pass through these two points are circles. Indeed, circles may be characterized by conics which pass through this conjugate pair.

  Note that the ``only if'' part is easy because the stereographic projection is conformal. For the converse, we needed to show that the pullback was of the same degree. 
\end{proof}

\begin{corollary}
  Any knot in $\mathbb{RP}^3$ that can be constructed inductively by using only circles, can be pulled back to via the stereographic projection, to a knot of the same degree in $S^3$. 
\end{corollary}

Figure~\ref{fig:figure8} shows an example of the figure-eight knot of degree~$8$ constructed in $S^3$. The knot is constructed in $\mathbb{RP}^3$ using only circles. Therefore, by the previous corollary, it is possible to lift it back to a degree~$8$ curve in the sphere. 

\begin{figure}[t]
	\label{fig:figure8}
	\begin{center}
		\includegraphics[height=3cm]{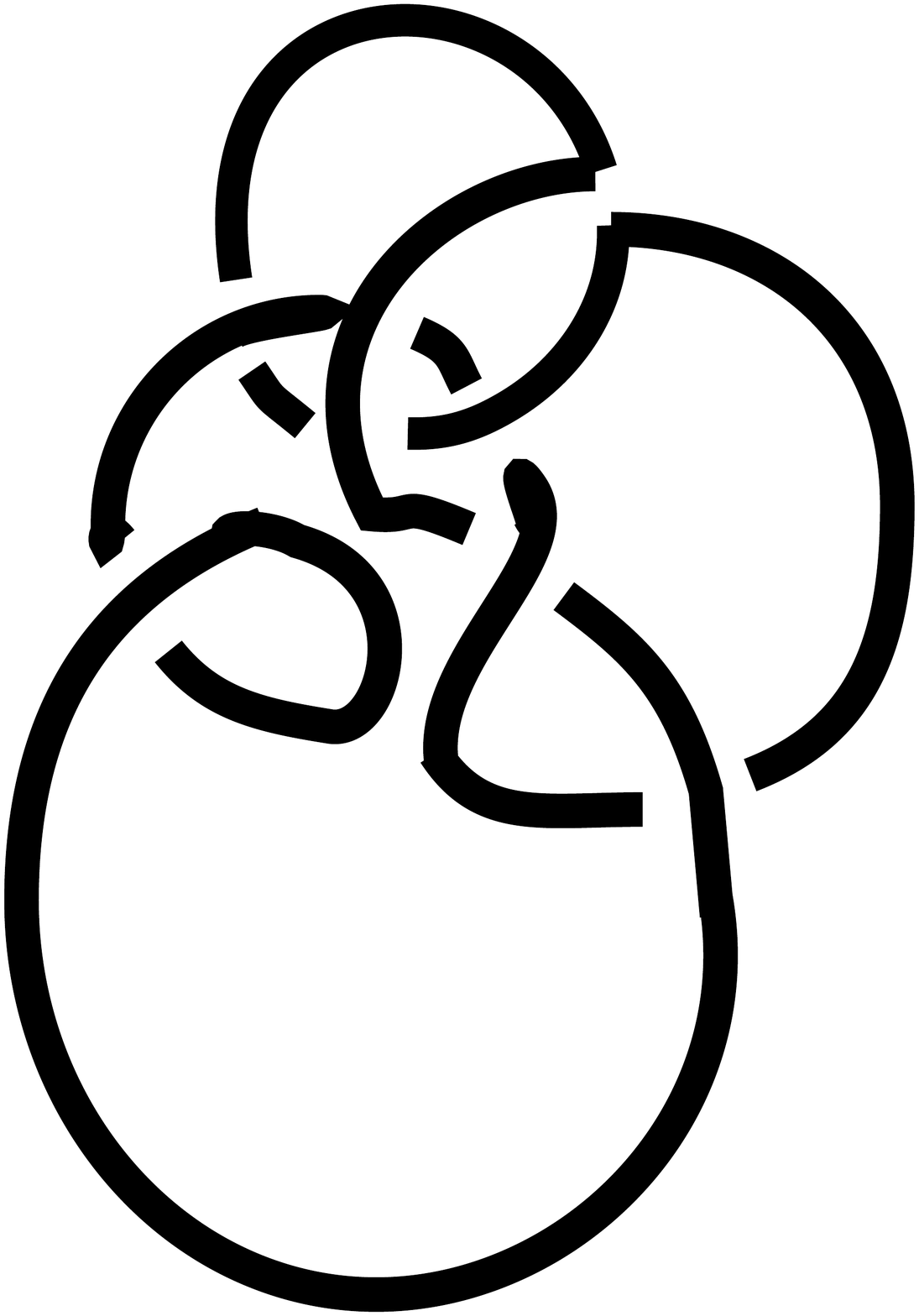}
	\caption{Figure-eight knot in $S^3$ realized by a degree~$8$ rational curve.}
	\end{center}
\end{figure}

The fact that we are allowed to use only circles (and not general conics) is a strong restriction. It prevents us from constructing the wall between the trefoil and figure eight knot as is demonstrated in the following section. The impossibility of this construction was already predicted by the classification.

\subsection{Example of a construction impossible in $S^3$}
\label{impossible}
\begin{center}
		\includegraphics[height=5cm]{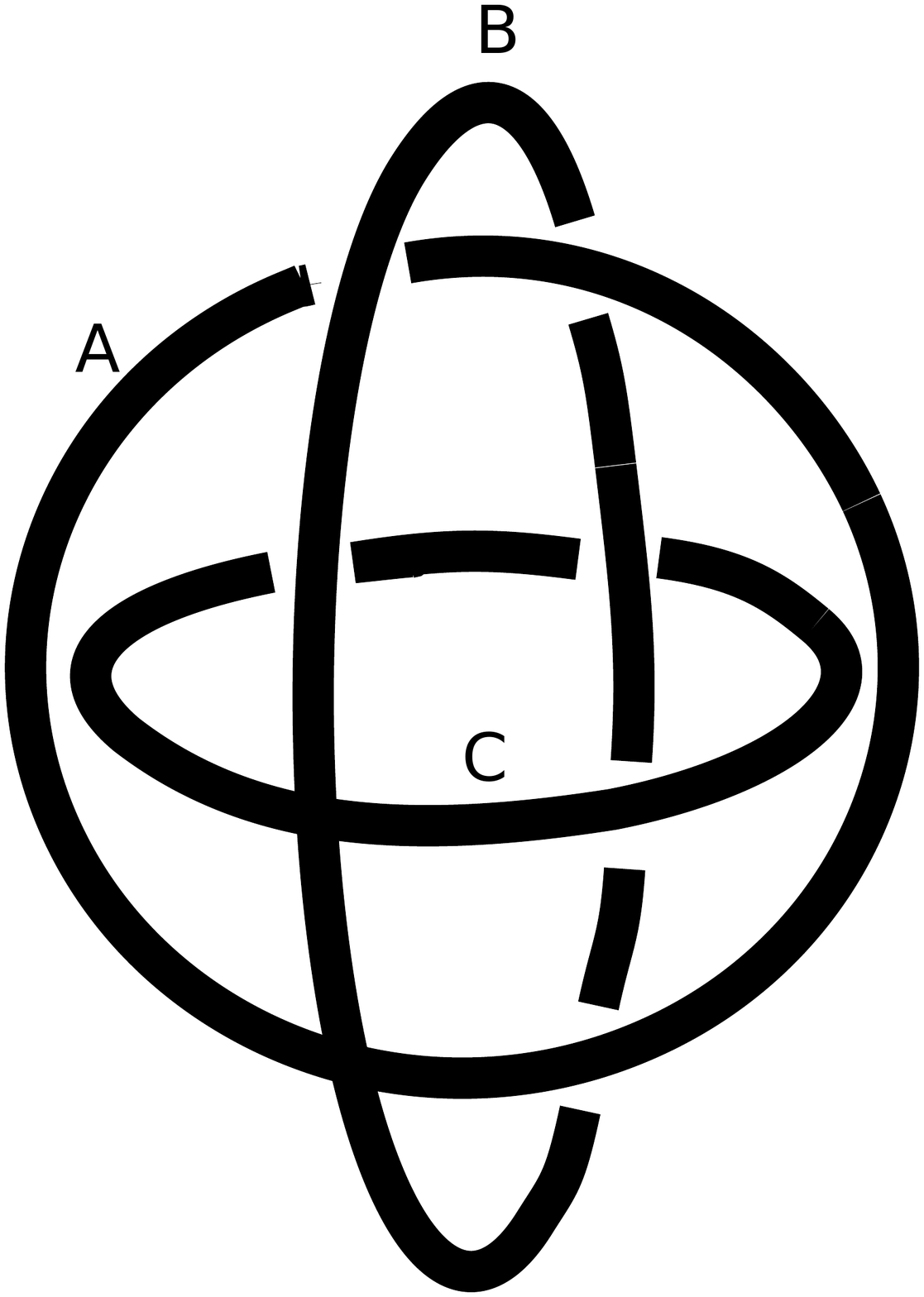}
	\end{center}
	\vspace{-1cm}
	In~\cite{bjorklund} constructed the figure of eight knot and the trefoil knot by perturbing the above union of ellipses. However, the figure eight knot should not be realizable by a degree~6 rational curve in $S^3$. The following remarks will show precisely why this construction fails in $S^3$. 
	
	If such a configuration came from the sphere via the stereographic projection, then  the conics would have to be circles because the only conics in the sphere are circles and the stereographic projection takes circles to circles.  
	
	In the figure, Circle B meets Circle C in the plane containing it but also meets that plane in another point. It will be shown that if Circle C is large enough to contain this other point in the interior component, it will be too large to remain unlinked with Circle A

		The following lemma will make it more helpful to consider the above configuration with a minor change: keeping the point of intersection of circle A and circle B fixed, circle A is shrunk to create another intersection point, like this: 		 
	\begin{center}
		\includegraphics[height=5cm]{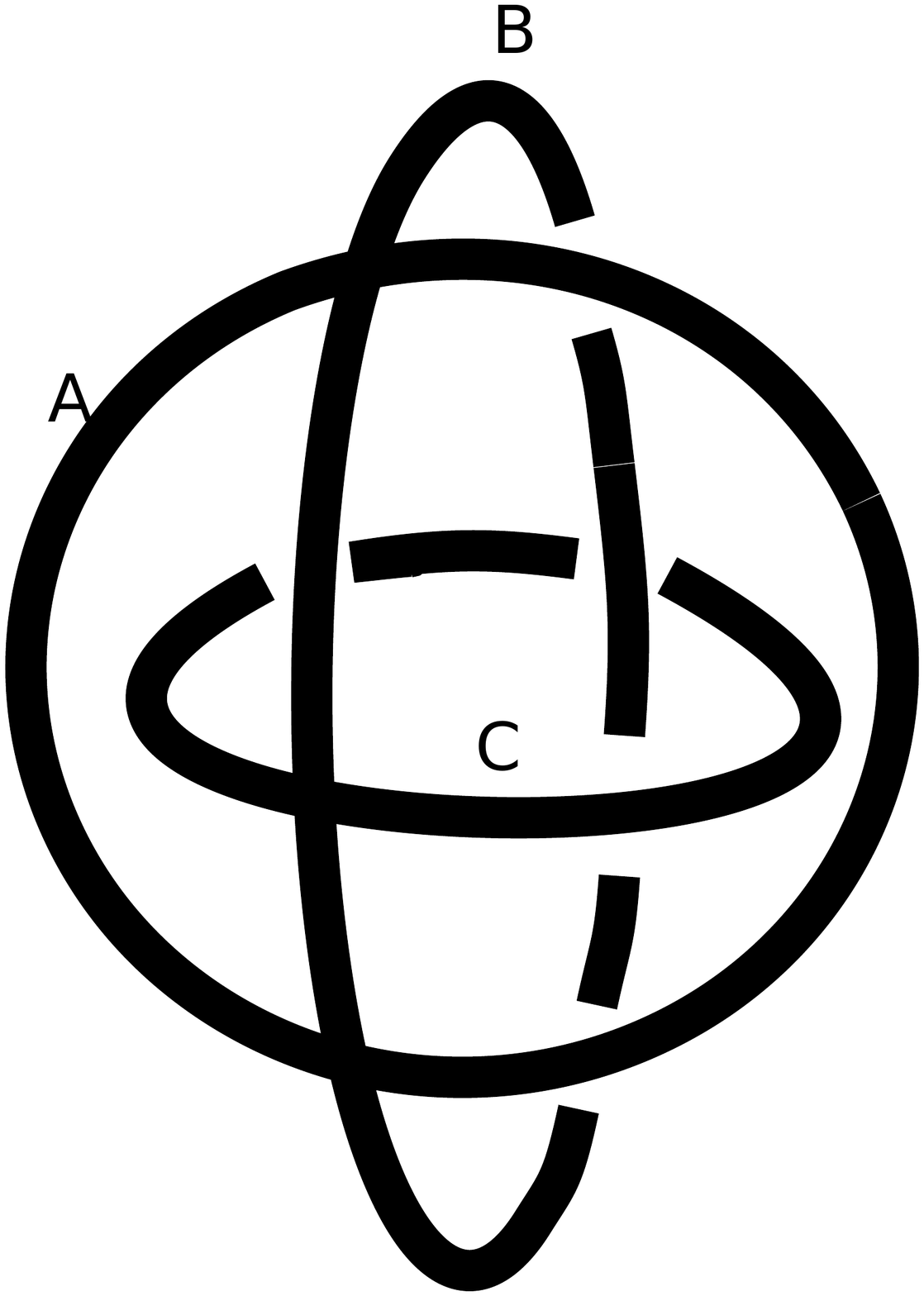}
	\end{center}
	\vspace{-1cm}
If the first configuration of circles is possible, so will the second configuration, but the second configuration will be easier to prove impossible because now the unions of circles A and B lie on a sphere:
\begin{lemma}
  The union of two circles in $\mathbb{RP}^3$ that intersect in two points, lies on a sphere. 
  \label{sphere}
\end{lemma}
\begin{proof}
  There is a sphere containing the two intersection points and one additional point from each of the two circles (since they are only four in number). The sphere shares three points with each circle and therefore contains them.
\end{proof}

 The plane containing circle C will intersect the sphere containing circles A and B in following way:
	\vspace{-1cm}
	\begin{center}
		\includegraphics[height=5cm]{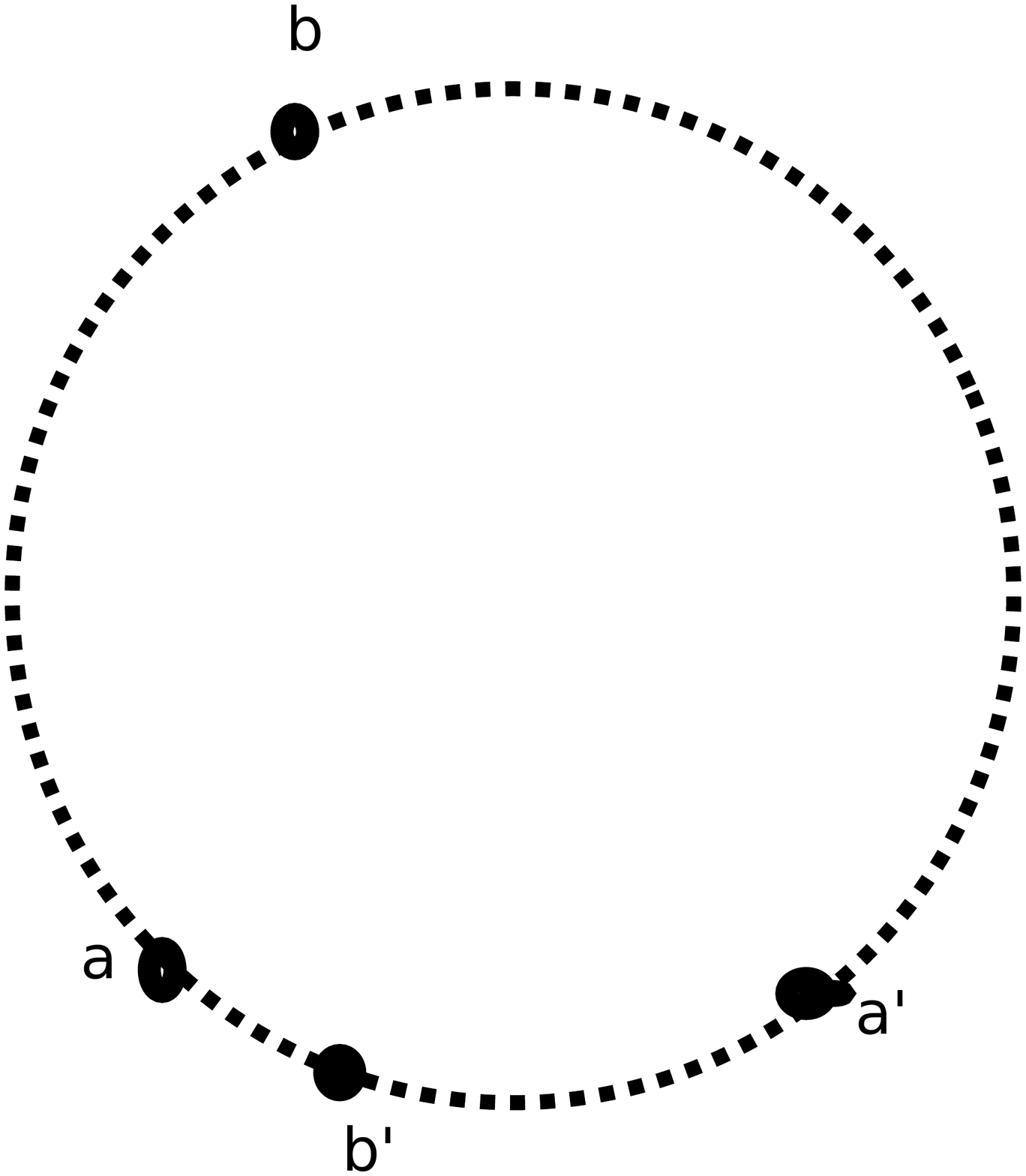}
	\end{center}
	\vspace{-0.8cm}
Here a and a$'$ are the intersections of circle A with the plane, while b and b$'$ are intersections of the circle B with the plane. The dotted circle is the intersection of the sphere with the plane. 

Points a and a$'$ divide the dotted circle into two arcs, one that contains b and the other that contains b$'$. 

Circle C passes through b$'$ and will intersect the dotted circle in one more point. If that point is on the arc aa$'$ containing the point b then circle C and circle A would have a non-zero linking number, which we know is not the case. Therefore both points of intersection of circle C with the dotted circle lie on the arc aa$'$ containing b$'$. The arc of the dotted circle that is in the interior of circle C must therefore be disjoint from the arc aa$'$ containing b. This proves that b cannot lie in the interior of circle C which is a contradiction.
	\begin{figure}[h]
	\begin{center}
		\includegraphics[height=4cm]{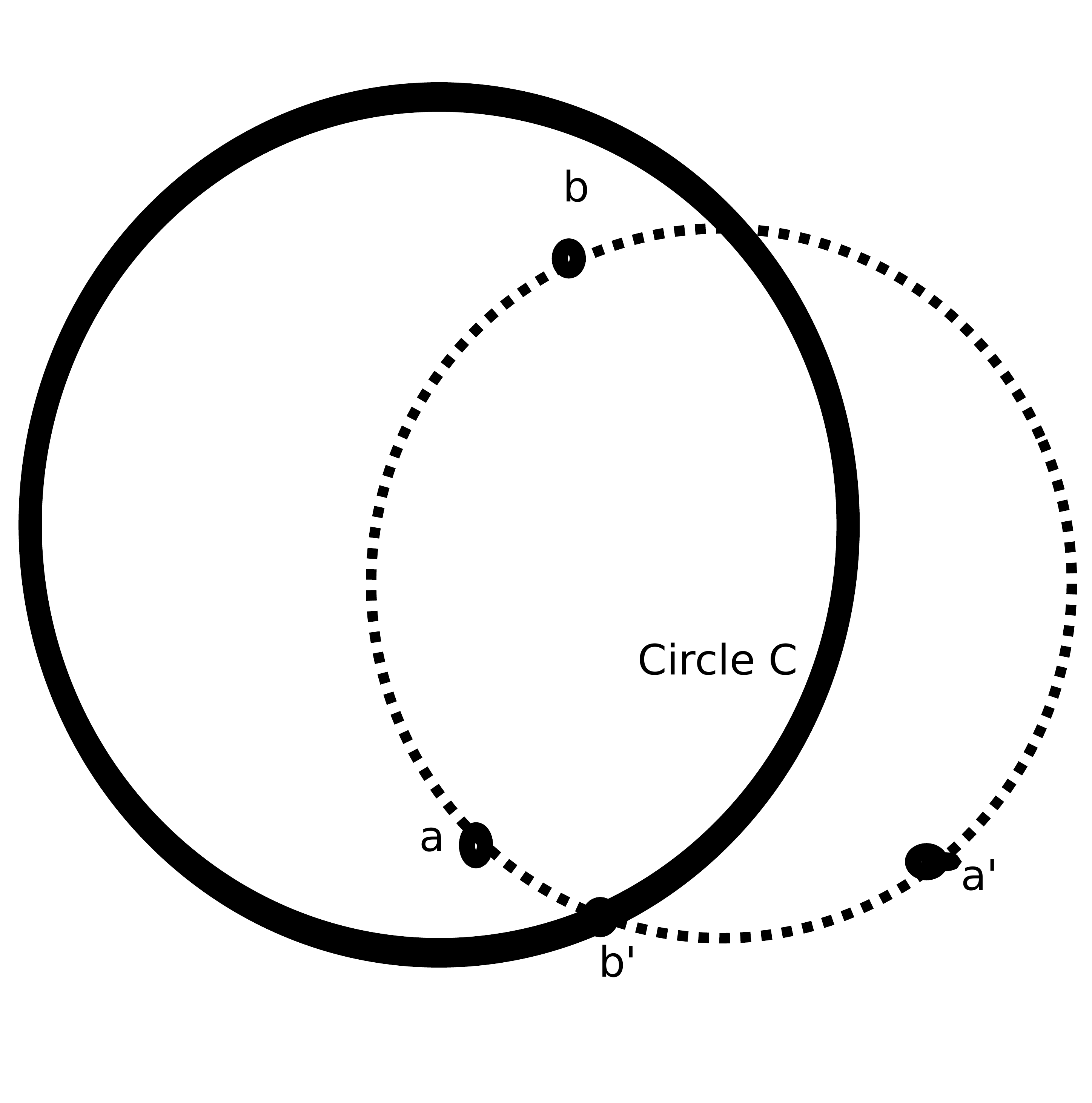}
		\includegraphics[height=4cm]{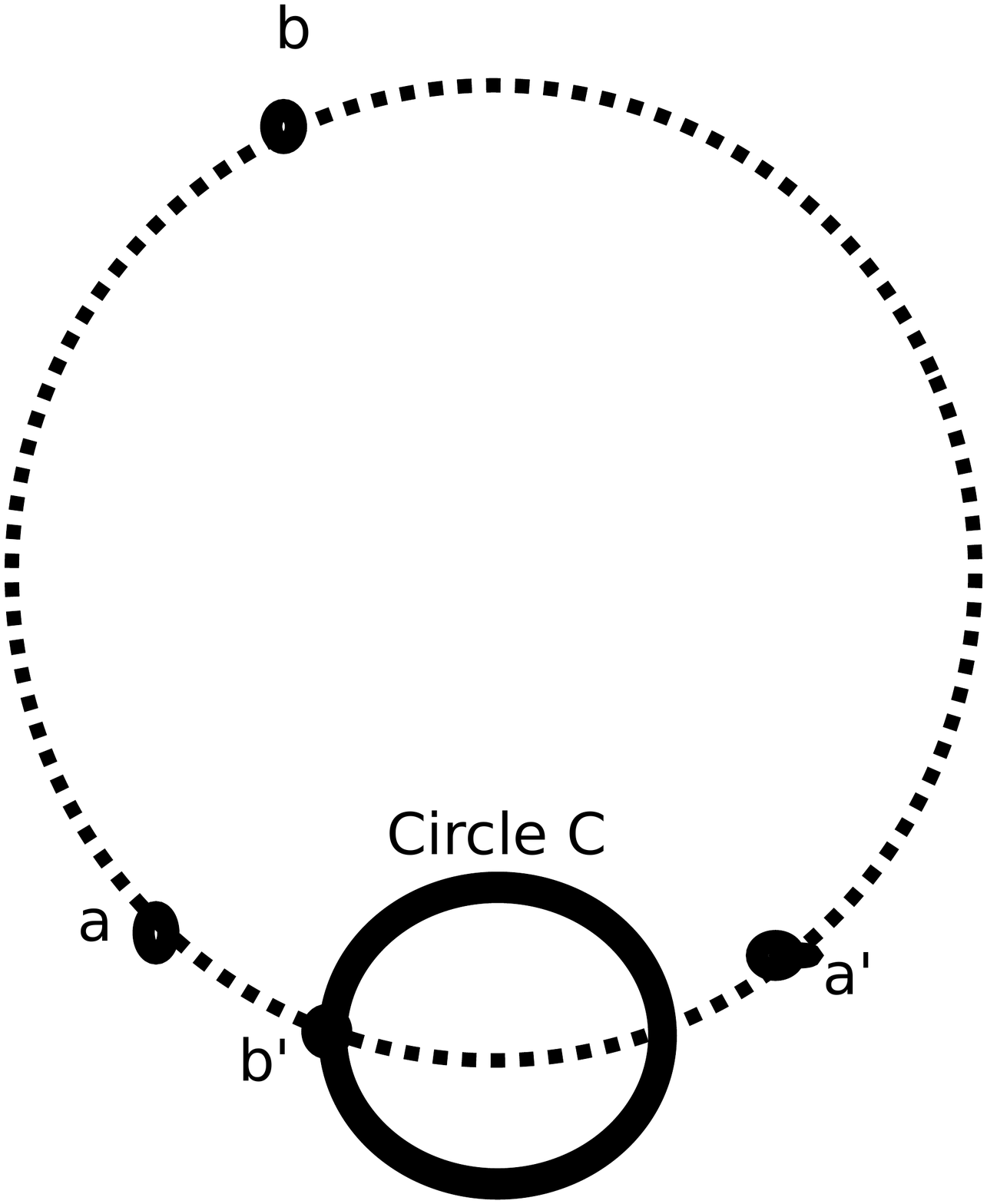}
	\end{center}
		\caption{Circle C will either link with circle A or not contain b. In the second case the arc contained in the interior of Circle C is disjoint from the arc aa$'$ containing b.}
	\end{figure}
	\section{Acknowledgements}
	I am grateful to Prof. Oleg Viro for suggesting this problem and for many valuable ideas, suggestions, and discussions. 
\bibliography{references}{}
\bibliographystyle{acm}
\end{document}